\documentclass[final]{siamonline1116} 

\usepackage{amsfonts,amssymb,amsmath}

\usepackage{array}
\newcolumntype{t}{>{\ttfamily}c}

\newcommand*\conj[1]{\bar{#1}}
\newcommand*\bigo{\mathcal{O}}
\newcommand*\C{\mathbb{C}}
\newcommand*\R{\mathbb{R}}

\DeclareMathOperator{\mat}{mat}

\newsiamremark{remark}{Remark}

\crefname{algorithm}{algorithm}{algorithms}
\crefname{remark}{remark}{remarks}

\title{Variable projection methods for an
  optimized dynamic mode decomposition
  \thanks{Submitted to the editors DATE.
    \funding{Air Force Office of Scientific Research
      (AFOSR) grant FA9550-15-1-0385.}}}

\author{Travis Askham
  \thanks{Department of Applied Mathematics,
    University of Washington, Seattle, WA
    (\email{askham@uw.edu}, \email{kutz@uw.edu}).}
v  \and
  J. Nathan Kutz
  \footnotemark[2]
}

\begin{document}

\maketitle

\begin{abstract}
The dynamic mode decomposition (DMD) has become a 
leading tool for data-driven modeling of dynamical 
systems, providing a regression framework for 
fitting linear dynamical models to time-series measurement 
data.  We present a simple algorithm for computing an 
optimized version of the DMD for data which may be collected 
at unevenly spaced sample times. By making use of the 
variable projection method for nonlinear least squares 
problems, the algorithm is capable of solving the underlying 
nonlinear optimization problem efficiently. We explore the 
performance of the algorithm with some numerical examples 
for synthetic and real data from dynamical systems and find 
that the resulting decomposition displays less bias in 
the presence of noise than standard DMD algorithms. Because 
of the flexibility of the algorithm, we also present some 
interesting new options for DMD-based analysis.   
\end{abstract}

\begin{keyword}
  dynamic mode decomposition, inverse linear systems,
  variable projection algorithm,
  inverse differential equations
\end{keyword}

\begin{AMS}
  37M02, 65P02, 49M02
\end{AMS}


\section{Introduction}

Suppose that ${\mathbf z}_j \in \C^n$ are snapshots of
a dynamical system $\dot{{\mathbf z}}(t) = {\mathbf f}({\mathbf z}(t))$
at equispaced times
$t_j = j \Delta t$. Let ${\mathbf A}$ be the best fit
\textit{linear} operator which maps each ${\mathbf z}_j$
to ${\mathbf z}_{j+1}$.
The dynamic mode decomposition (DMD) is defined
as the set of eigenvector, eigenvalue pairs of
${\mathbf A}$. The DMD is then a way of decomposing
the data into dominant modes, each with an associated
frequency of oscillation and rate of growth/decay.
This is an alternative decomposition to the 
proper orthogonal decomposition (POD): whereas the
DMD provides dynamical information about the system
but the modes are not orthogonal, the POD provides 
orthogonal modes but no dynamical information.
As such, the DMD is an enabling data-driven 
modeling strategy since it provides a best-fit, linear 
characterization of a nonlinear dynamical system from data alone.  

The DMD has its roots in the fluid dynamics community,
where it was applied to the analysis of numerical simulations
and experimental data of fluid flows
\cite{rowley2009,schmid2010}. 
Over the past decade, its popularity has grown and it
has been applied as a diagnostic tool,
as a means of model order reduction, and as a component
of optimal controller design for a variety of 
dynamical systems.
The DMD also has connections to the Koopman spectral
analysis of nonlinear dynamical systems, a line of
inquiry which has been pursued in, inter alia,
\cite{rowley2009,bagheri2013,mezic2013,tu2013}. In
particular, the DMD shows promise as a tool for the
analysis of general nonlinear systems. We will not
stress this aspect here but rather focus on the
DMD as an algorithm for approximating data by a linear
system.

A well-studied pitfall of the DMD is that 
the computed eigenvalues are biased by the presence
of sensor noise \cite{hemati2015,dawson2016}. Intuitively, this
is a result of the fact that the standard algorithms
treat the data pairwise, i.e. snapshot to snapshot rather
than as a whole, and favor one direction (forward 
in time). In \cite{dawson2016}, Dawson et al.
present several methods for debiasing within
the standard DMD framework. These methods have the advantage
that they can be computed with essentially the same set of
robust and fast tools as the standard DMD.

As an alternative, the \textit{optimized DMD} of
\cite{chen2012} treats all of the snapshots of
the data at once. This avoids much of the bias
of the original DMD but requires the solution of
a (potentially) large nonlinear optimization
problem.
It is believed that the ``nonconvexity
of the optimization [required for the optimized DMD]
potentially limits its utility''
\cite{dawson2016} but the results of this paper
suggest that the optimized DMD should be the DMD
algorithm of choice in many settings.
We will present some efficient algorithms for computing
the optimized DMD and discuss its properties.

The primary computational tool at the heart of these
algorithms is the variable 
projection method \cite{golub1973}.
To apply variable projection, the DMD is rephrased as a problem 
in exponential data fitting (specifically, inverse differential
equations), an area of research 
which has been extensively developed and has many applications 
\cite{golub1979,pereyra2010}. The variable projection
method leverages the special structure of the 
exponential data fitting problem, so that many of the 
unknowns may be eliminated from the optimization.
An additional benefit of these tools is that the 
snapshots of data no longer need to be taken at 
regular intervals, i.e. the sample times do not 
need to be equispaced. We suggest a pair of algorithms,
each a modified version of the original algorithm of \cite{golub1979}, 
for computing the optimized DMD and an initialization scheme
based on the standard DMD. 

The rest of this paper is organized as follows. In
\cref{sec:background}, we present some of
the relevant preliminaries of variable projection and
the DMD. In \cref{sec:algorithm}, we present the
definition, algorithms, and an initialization scheme
for the optimized DMD.  In \cref{sec:examples}, we
demonstrate
the low inherent bias of the algorithm in the presence of noise
on some simple examples and present some applications of the 
method to both synthetic and real data sets, some with
snapshots whose sample times are unevenly spaced. The final 
section contains some concluding thoughts and ideas
for further research.

\section{Preliminaries} \label{sec:background}

\subsection{Notation} \label{subsec:notation}

Throughout this paper we use mostly standard notation,
with some MATLAB style notation for convenience.
Matrices are typically denoted by bold, capital letters
and vectors by bold, lower-case. Let ${\mathbf A}$ and
${\mathbf B}$ be matrices
and ${\mathbf v}$ a vector of length $m$. Then

\begin{itemize}
\item $v_i$ denotes the $i$th entry of ${\mathbf v}$;
\item $A_{i,j}$ denotes the entry in the $i$th row
  and $j$th column of ${\mathbf A}$;
\item ${\mathbf v}_i$ denotes the $i$th vector in a
  sequence of vectors;
\item $\conj{{\mathbf A}}$ denotes the entrywise complex conjugate
  of ${\mathbf A}$;
\item ${\mathbf A}^\intercal$ denotes the transpose of
  ${\mathbf A}$, which satisfies
  $A^\intercal_{i,j} = A_{j,i}$;
\item ${\mathbf A}^*$ denotes the complex conjugate
  transpose of ${\mathbf A}$, which satisfies
  ${\mathbf A}^* = \conj{{\mathbf A}^\intercal}$;
\item ${\mathbf A}^\dagger$ denotes the Moore-Penrose pseudoinverse
  of ${\mathbf A}$, which satisfies ${\mathbf A}{\mathbf A}^\dagger {\mathbf A} = {\mathbf A}$,
  ${\mathbf A}^\dagger {\mathbf A} {\mathbf A}^\dagger = {\mathbf A}^\dagger$,
  $({\mathbf A}{\mathbf A}^\dagger)^* = {\mathbf A}{\mathbf A}^\dagger$, and $({\mathbf A}^\dagger {\mathbf A})^* = {\mathbf A}^\dagger {\mathbf A}$;
\item ${\mathbf A}(i_1  :  i_2,j_1  :  j_2)$ denotes the submatrix corresponding
  to the $i_1$th through $i_2$th rows and $j_1$th through $j_2$th
  columns of ${\mathbf A}$.
\item ${\mathbf A}( :  ,j)$ denotes the vector given by the $j$th
  column of ${\mathbf A}$;
\item ${\mathbf A}( : )$ denotes the vector which results by stacking
  all of the columns of ${\mathbf A}$, taken in order;
\item ${\mathbf A} = \mbox{diag}({\mathbf v})$ denotes the square matrix of
  size $m\times m$ which satisfies $A_{i,i} = v_i$ and
  $A_{i,j} = 0$ for $i\neq j$;
\item and ${\mathbf A} \otimes {\mathbf B}$ denotes the Kronecker of ${\mathbf A}$ and ${\mathbf B}$, e.g.
  if ${\mathbf A}$ is $2\times 2$ then

  \begin{equation}
    {\mathbf A}\otimes {\mathbf B} = \begin{pmatrix}
      A_{1,1} {\mathbf B} & A_{1,2} {\mathbf B} \\
      A_{2,1} {\mathbf B} & A_{2,2} {\mathbf B}
    \end{pmatrix} \; .
  \end{equation}
\end{itemize}

\subsection{Variable projection} \label{subsec:varpro}

In this section, we will review some details
of the classical variable projection algorithms
\cite{golub1973,kaufman1975,golub1979,golub2003,oleary2013}
which are relevant to the optimized DMD and our
implementation of the method. We also include a
brief coda concerning modern advances in the
variable projection framework.

\subsubsection{Nonlinear least squares}

The variable projection algorithm was
originally conceived for the solution of separable
nonlinear least squares problems. The vector version
of a separable least squares problem is of the form

\begin{equation} \label{eq:snlls}
  \mbox{minimize} \| {\mathbf {\boldsymbol \eta}} - {\boldsymbol \Phi}({\boldsymbol \alpha}) {\boldsymbol \beta} \|_2
  \qquad \mbox{ over } {\boldsymbol \alpha} \in \C^k, {\boldsymbol \beta} \in \C^l \; ,
\end{equation}
where ${\mathbf {\boldsymbol \eta}} \in \C^m$, ${\boldsymbol \Phi}({\boldsymbol \alpha}) \in \C^{m\times l}$,
and $m > l$.
A typical example of such a problem is the approximation
of a function $\eta(t)$ by a linear combination of $l$
nonlinear functions $\phi_j({\boldsymbol \alpha},t)$ with coefficients $\beta_j$.
In this case, we set $\eta_i = \eta(t_i)$ and
$\Phi({\boldsymbol \alpha})_{i,j} = \phi_j({\boldsymbol \alpha},t_i)$
 for $m$ sample times $t_i$.
Here, and in the remainder of the paper, the dependence of
${\boldsymbol \Phi}({\boldsymbol \alpha})$ on the times
$t_i$ is implicit.

The key to the variable projection algorithm is
the following observation: for a fixed ${\boldsymbol \alpha}$, the
${\boldsymbol \beta}$ which minimizes $\| {\boldsymbol \eta} - {\boldsymbol \Phi}({\boldsymbol \alpha}) {\boldsymbol \beta} \|_2$
is given by ${\boldsymbol \beta} = {\boldsymbol \Phi}({\boldsymbol \alpha})^\dagger {\boldsymbol \eta}$. With this
observation, we can rewrite the minimization problem
\cref{eq:snlls} in terms of ${\boldsymbol \alpha}$ alone, solve
for the minimizer $\hat{{\boldsymbol \alpha}}$, and recover the
coefficients ${\boldsymbol \beta}$ corresponding to this minimizer via
$\hat{{\boldsymbol \beta}} = {\boldsymbol \Phi}(\hat{{\boldsymbol \alpha}})^\dagger {\boldsymbol \eta}$.
It is clear that the minimization problem in ${\boldsymbol \alpha}$
alone is equivalent to

\begin{equation} \label{eq:snllsproj}
  \mbox{minimize} \dfrac12 \| {\boldsymbol \eta} - {\boldsymbol \Phi}({\boldsymbol \alpha}) {\boldsymbol \Phi}({\boldsymbol \alpha})^\dagger {\boldsymbol \eta} \|_2^2
  \qquad \mbox{ over } {\boldsymbol \alpha} \in \C^k \; ,
\end{equation}
where we have squared the error and rescaled for
notational convenience.

Typically \cite{kaufman1975,golub1979,golub2003,oleary2013},
the Levenberg-Marquardt algorithm
\cite{levenberg1944,marquardt1963} is used for
the solution of the new minimization problem
\cref{eq:snllsproj}.
This is an iterative procedure for solving
\cref{eq:snllsproj}, which
produces a sequence of vectors ${\boldsymbol \alpha}_i$ which
should converge to a nearby (local) minimizer.
Let

\begin{equation}
  {\boldsymbol \rho}({\boldsymbol \alpha}) = {\boldsymbol \eta} - {\boldsymbol \Phi}({\boldsymbol \alpha}) {\boldsymbol \Phi}({\boldsymbol \alpha})^\dagger {\boldsymbol \eta}
\end{equation}
denote the residual.
We will use ${\boldsymbol \delta}_i$ to denote the update to ${\boldsymbol \alpha}_i$,
so that ${\boldsymbol \alpha}_{i+1} = {\boldsymbol \alpha}_i - {\boldsymbol \delta}_i$ and assume
that a parameter $\nu_i$ is specified at each iteration. The
Levenberg-Marquardt update is defined to be the solution of

\begin{equation}
  \mbox{minimize} \left  \|
  \begin{pmatrix} {\mathbf J}({\boldsymbol \alpha}_i) \\ \nu_i {\mathbf M}({\boldsymbol \alpha}_i) 
  \end{pmatrix} {\boldsymbol \delta}_i - \begin{pmatrix} {\boldsymbol \rho}({\boldsymbol \alpha}_i) \\ 0
  \end{pmatrix}
  \right \|_2^2
  \qquad \mbox{ over } {\boldsymbol \delta}_i \in \C^k \; ,
\end{equation}
where ${\mathbf J}({\boldsymbol \alpha}_i)$ is the Jacobian of ${\boldsymbol \rho}({\boldsymbol \alpha})$
evaluated at ${\boldsymbol \alpha}_i$ and ${\mathbf M}({\boldsymbol \alpha}_i)$ is a
diagonal matrix of scalings such that $M({\boldsymbol \alpha}_i)_{jj} =
\|{\mathbf J}({\boldsymbol \alpha}_i)( : ,j)\|_2$. Typically the parameter
$\nu_i$ is chosen as part of a trust-region method,
i.e. $\nu_i$ is increased until a step is found so that
the new ${\boldsymbol \alpha}_{i+1}$ results in a smaller residual.
When possible, the parameter $\nu_i$ is reduced so that
the update is more like a standard Gauss-Newton update,
which results in a fast convergence rate. 
Ruhe and Wedin \cite{ruhe1980} showed that
when superlinear convergence occurs
for Gauss-Newton applied to the original
problem \cref{eq:snlls}, then it also occurs for the
projected problem \cref{eq:snllsproj}.
See \cite{marquardt1963,osborne1976} for more detail
on choosing $\nu_i$ and the overall structure
of the Levenberg-Marquardt method.

In order to apply this method, we must have an expression
for the Jacobian of ${\boldsymbol \rho}({\boldsymbol \alpha})$. Typically,
the derivatives of ${\boldsymbol \Phi}$ with respect to ${\boldsymbol \alpha}$
are known analytically, e.g. they are simple to
obtain in the case that $\phi_j({\boldsymbol \alpha},t) =
\exp( \alpha_j t)$.
We therefore assume that these derivatives are available.
In the following, we will leave out the dependence of
the matrices on ${\boldsymbol \alpha}$ in order to simplify the
notation.
Let ${\mathbb P}_{ \Phi}$ denote the orthogonal projection onto
the columns of ${\boldsymbol \Phi}$, i.e. ${\mathbb P}_{ \Phi} = {\boldsymbol \Phi} {\boldsymbol \Phi}^\dagger$,
and ${\mathbb P}_{ \Phi}^\perp$ denote the projection onto the complement
of the column space of ${\boldsymbol \Phi}$, i.e.
${\mathbb P}_{ \Phi}^\perp = {\mathbf I} - {\boldsymbol \Phi} {\boldsymbol \Phi}^\dagger$. Note that
${\boldsymbol \rho} = {\mathbb P}_{ \Phi}^\perp {\boldsymbol \eta}$. From Lemma 4.1 of \cite{golub1973},
we have

\begin{equation}
{\mathbf J}( : ,j) =  \frac{\partial {\boldsymbol \rho}}{\partial \alpha_j} =
  - \left ( {\mathbb P}_{ \Phi}^\perp \dfrac{\partial {\boldsymbol \Phi}}{\partial \alpha_j}
  {\boldsymbol \Phi}^\dagger + \left ({\mathbb P}_{ \Phi}^\perp \dfrac{\partial {\boldsymbol \Phi}}{\partial \alpha_j}
  {\boldsymbol \Phi}^\dagger \right )^* \right ) {\boldsymbol \eta} \; . 
\end{equation}
Kaufman \cite{kaufman1975} recommends the approximation

\begin{equation}
{\mathbf J}( : ,j) =  \frac{\partial {\boldsymbol \rho}}{\partial \alpha_j} \approx
  - {\mathbb P}_{ \Phi}^\perp \dfrac{\partial {\boldsymbol \Phi}}{\partial \alpha_j}
  {\boldsymbol \Phi}^\dagger {\boldsymbol \eta} \; , \label{eq:kauf}
\end{equation}
which is accurate for small residuals. This
approximation is used in \cite{golub1979}
and there is
some debate over whether this approximation to the
Jacobian is superior to the full expression,
see, inter alia, \cite{mullen2007,oleary2013}.
In our MATLAB implementation \cite{askham2017optdmd}, we have used the
full expression.

All of the terms in the above expression can be computed
by making use of the singular value decomposition
(SVD) of ${\boldsymbol \Phi}$. Let $q$ be the rank of ${\boldsymbol \Phi}$.
The (reduced) SVD of a matrix ${\boldsymbol \Phi}$ provides three
matrices ${\mathbf U}$, ${\boldsymbol \Sigma}$, and ${\mathbf V}$ such that ${\boldsymbol \Phi} = {\mathbf U}{\boldsymbol \Sigma} {\mathbf V}^*$,
${\mathbf U} \in \C^{m\times q}$ and ${\mathbf V} \in \C^{l\times q}$ have orthonormal columns,
and ${\boldsymbol \Sigma} \in \R^{q\times q}$ is diagonal with nonnegative
entries.
Given the SVD of ${\boldsymbol \Phi}$, we have that

\begin{equation}
  {\mathbb P}_{ \Phi}^\perp \dfrac{\partial {\boldsymbol \Phi}}{\partial \alpha_j}
  {\boldsymbol \Phi}^\dagger {\boldsymbol \eta} = ({\mathbf I}-{\mathbf U}{\mathbf U}^*) \dfrac{\partial {\boldsymbol \Phi}}{\partial \alpha_j} {\boldsymbol \beta} \; ,
\end{equation}
where we have solved for ${\boldsymbol \beta}$ using
${\boldsymbol \beta} = {\mathbf V}{\boldsymbol \Sigma}^{-1}{\mathbf U}^* {\boldsymbol \eta}$, and

\begin{equation}
  \left ({\mathbb P}_{ \Phi}^\perp \dfrac{\partial {\boldsymbol \Phi}}{\partial \alpha_j}
  {\boldsymbol \Phi}^\dagger \right )^* {\boldsymbol \eta} = {\mathbf U}{\boldsymbol \Sigma}^{-1} {\mathbf V}^*\dfrac{\partial {\boldsymbol \Phi}}{\partial \alpha_j}^* {\boldsymbol \rho} \; ,
\end{equation}
where we have used the fact that
$({\mathbb P}_{ \Phi}^\perp)^* {\boldsymbol \eta} = {\mathbb P}_{ \Phi}^\perp {\boldsymbol \eta} = {\boldsymbol \rho}$.

%
%

\subsubsection{Variable projection for multiple right hand 
sides} \label{subsec:mrhs}

One of the primary innovations of \cite{golub1979}
was the extension of the variable projection method
developed above to the case of
multiple right hand sides, i.e. to the problem 

\begin{equation} \label{eq:snllsmrhs}
  \mbox{minimize} \| {\mathbf H} - {\boldsymbol \Phi}({\boldsymbol \alpha}) {\mathbf B} \|_F
  \qquad \mbox{ over } {\boldsymbol \alpha} \in \C^k, {\mathbf B} \in \C^{l\times n} \; ,
\end{equation}
where ${\mathbf H} \in \C^{m\times n}$, ${\boldsymbol \Phi}({\boldsymbol \alpha}) \in \C^{m\times l}$,
and $m>l$. A typical example of such a problem is the 
approximation of $n$ functions $\eta_p(t)$, each by a linear combination
of $l$ nonlinear functions $\phi_j({\boldsymbol \alpha},t)$ with coefficients
$B_{j,p}$.
In this case, we have $H_{i,p} = \eta_p(t_i)$ and
$ \Phi({\boldsymbol \alpha})_{i,j}$ is as before, ${ \Phi}({\boldsymbol \alpha})_{i,j} = \phi_j({\boldsymbol \alpha},t_i)$.
We note that the vector of parameters
${\boldsymbol \alpha}$ is the same for 
each function $\eta_p$, so that the problem is coupled.

This problem can be solved efficiently using ideas similar to
those outlined for the case of a single right hand side above.
In order to use the same language as for the vector case,
we need to reshape problem \cref{eq:snllsmrhs}.
Let ${\boldsymbol \eta} = {\mathbf H}( : )$ and ${\boldsymbol \beta} = {\mathbf B}( : )$. Then \cref{eq:snllsmrhs}
is equivalent to

\begin{equation} \label{eq:snllsmrhs_equiv}
  \mbox{minimize} \| {\boldsymbol \eta} - {\mathbf I}_n \otimes {\boldsymbol \Phi}({\boldsymbol \alpha}) {\boldsymbol \beta} \|_2
  \qquad \mbox{ over } {\boldsymbol \alpha} \in \C^k, {\boldsymbol \beta} \in \C^{ln} \; .
\end{equation}
For a given ${\boldsymbol \alpha}$, the matrix ${\mathbf B}$ is given by
${\boldsymbol \Phi}^\dagger {\mathbf H}$ so that the computation of ${\boldsymbol \beta}$ can be
done in blocked form. Likewise, the computation of
${\boldsymbol \rho}$ can be blocked. Let ${\mathbf P} = {\mathbf H}-{\boldsymbol \Phi} {\mathbf B}$. Then ${\boldsymbol \rho} = {\mathbf P}( : )$.
Importantly, the formation of the Jacobian can also
be blocked. If we set

\begin{equation}
  {\mathbf J}^{\mat}_j = - \left ( {\mathbb P}_{ \Phi}^\perp \dfrac{\partial {\boldsymbol \Phi}}{\partial \alpha_j}
  {\boldsymbol \Phi}^\dagger + \left ({\mathbb P}_{ \Phi}^\perp \dfrac{\partial {\boldsymbol \Phi}}{\partial \alpha_j}
  {\boldsymbol \Phi}^\dagger \right )^* \right ) {\mathbf H} \; ,
\end{equation}
then ${\mathbf J}( : ,j) = {\mathbf J}^{\mat}_j( : )$. As above, if the SVD of
${\boldsymbol \Phi}$ is computed, we may write

\begin{equation}
  {\mathbb P}_{ \Phi}^\perp \dfrac{\partial {\boldsymbol \Phi}}{\partial \alpha_j}
  {\boldsymbol \Phi}^\dagger {\mathbf H} = ({\mathbf I}-{\mathbf U}{\mathbf U}^*) \dfrac{\partial {\boldsymbol \Phi}}{\partial \alpha_j} {\mathbf B} \; ,
  \label{eq:jac1blocked}
\end{equation}
where we have solved for ${\mathbf B}$ using
${\mathbf B} = {\mathbf V}{\boldsymbol \Sigma}^{-1}{\mathbf U}^* {\mathbf H}$, and

\begin{equation}
  \left ({\mathbb P}_{ \Phi}^\perp \dfrac{\partial {\boldsymbol \Phi}}{\partial \alpha_j}
  {\boldsymbol \Phi}^\dagger \right )^* {\mathbf H} = {\mathbf U}{\boldsymbol \Sigma}^{-1} {\mathbf V}^*\dfrac{\partial {\boldsymbol \Phi}}{\partial \alpha_j}^* {\mathbf P} \; ,
  \label{eq:jac2blocked}
\end{equation}
where we have used the fact that
$({\mathbb P}_{ \Phi}^\perp)^* {\mathbf H} = {\mathbb P}_{ \Phi}^\perp {\mathbf H} = {\mathbf P}$.

Because this is the version of the variable projection algorithm
used for the computations in this paper, we will briefly discuss
its computational cost. The matrices of partial derivatives of
${\boldsymbol \Phi}$, i.e. 

\begin{equation}
{\mathbf D}_j  = \dfrac{\partial {\boldsymbol \Phi}}{\partial \alpha_j} \nonumber \; ,
\end{equation}
are often sparse in applications. As this is the case in our
application, we will make the simplifying assumption that
these matrices have one nonzero column. In our implementation,
these matrices are stored in MATLAB sparse matrix format,
though there are possibly more efficient ways to leverage the
sparsity, see \cite{oleary2013} for an example. In what follows,
we assume that MATLAB handling of sparse matrix-matrix multiplication
is optimal in the sense of operation count. Another
simplifying assumption we make is that $l=k$ and that ${\boldsymbol \Phi}$ is
full rank, i.e. $q = k$.

Each iteration of the algorithm is dominated by the cost of
forming the Jacobian, ${\mathbf J}$, and solving for the update, ${\boldsymbol \delta}$.
We have that ${\mathbf J} \in \C^{mn\times k}$ and ${\mathbf M}\in \C^{k\times k}$ so
that the solve for ${\boldsymbol \delta}$ is $\bigo (k^2mn)$ using standard
linear least squares methods, e.g. a QR factorization.
We will consider the cost of computing ${\mathbf J}$ in four steps: \\

\begin{enumerate}
\item the cost of the SVD of ${\boldsymbol \Phi}$,
\item forming ${\mathbf B}$ and ${\mathbf P}$,
\item applying formula \cref{eq:jac1blocked},
\item and applying formula \cref{eq:jac2blocked}. \\
\end{enumerate}

For step 1, the
SVD of ${\boldsymbol \Phi}$ costs $\bigo (k^2m)$ to compute with standard
methods. In step 2, ${\mathbf B}$ and ${\mathbf P}$ are formed via matrix-matrix
multiplications which are $\bigo( kmn)$. Note that steps 1
and 2 are computed once.

In step 3, the order of operations is more important.
We rewrite \cref{eq:jac1blocked} as

\begin{equation}
  {\mathbb P}_{ \Phi}^\perp \dfrac{\partial {\boldsymbol \Phi}}{\partial \alpha_j}
  {\boldsymbol \Phi}^\dagger {\mathbf H} = \left ( {\mathbf D}_j
  - {\mathbf U} \left ( {\mathbf U}* {\mathbf D}_j \right )  \right ) {\mathbf B} \; ,
\end{equation}
where the parentheses determine the order of the matrix
multiplications. Forming ${\mathbf U}^*{\mathbf D}_j$ costs $\bigo (km)$
and is itself sparse with one nonzero column. The
cost of forming ${\mathbf U}({\mathbf U}^*{\mathbf D}_j)$ is then again $\bigo (km)$
and is sparse with one nonzero column. Finally,
${\mathbf D}_j - {\mathbf U}({\mathbf U}^*{\mathbf D}_j)$ is still sparse with one nonzero column
so that the last multiplication giving $({\mathbf D}_j - {\mathbf U}({\mathbf U}^*{\mathbf D}_j)){\mathbf B}$
costs $\bigo (mn)$. Repeating these calculations for
each column is then $\bigo (k^2 m + kmn)$ in total.

In step 4, the order of operations are again important.
We rewrite \cref{eq:jac2blocked} as

\begin{equation}
    \left ({\mathbb P}_{ \Phi}^\perp \dfrac{\partial {\boldsymbol \Phi}}{\partial \alpha_j}
  {\boldsymbol \Phi}^\dagger \right )^* {\mathbf H} = {\mathbf U} ({\boldsymbol \Sigma}^{-1} ({\mathbf V}^*({\mathbf D}_j^* {\mathbf P}))) \; .
\end{equation}
Because ${\mathbf D}_j^*$ has one nonzero row, forming ${\mathbf D}_j^*{\mathbf P}$ is
$\bigo (mn)$ and the result has one nonzero row. Similarly,
it then costs $\bigo (kn)$ to form ${\mathbf V}^*({\mathbf D}_j^* {\mathbf P})$ but the result
is a full matrix of size $k \times n$. The product
${\boldsymbol \Sigma}^{-1}({\mathbf V}^*({\mathbf D}_j^*{\mathbf P}))$ simply scales the rows, which costs
$\bigo(kn)$. Finally, forming ${\mathbf U}({\boldsymbol \Sigma}^{-1}({\mathbf V}^*({\mathbf D}_j^*{\mathbf P})))$
is a dense matrix-matrix multiplication which costs
$\bigo(kmn)$. Repeating these calculations for each column
is then $\bigo (k^2n + k^2mn)$ in total. This unfavorable
scaling, when compared with that of step 3, is part of the
appeal of using the approximation \cref{eq:kauf}.

In the discussion of the optimized DMD, we will take for granted 
the existence of an algorithm for solving \cref{eq:snllsmrhs}.
See, for instance, the original Fortran implementation (follow
the URL in \cite{golub1979}). We have also prepared a 
MATLAB implementation for the computations in this manuscript
\cite{askham2017optdmd}.

\subsubsection{Inverse differential equations}
\label{sec:invdiffeq}

In \cite{golub1979}, it is observed that the inverse differential
equations problem can be phrased as a nonlinear least squares
problem with multiple right hand sides. Suppose that ${\mathbf z}(t) \in \C^n$
is the solution of 

\begin{equation}
  \dot{ {\mathbf z}}(t) = {\mathbf A} {\mathbf z}(t) \; ,
\end{equation}
with the initial condition ${\mathbf z}(0) = {\mathbf z}_0$. The solution of this
problem is known analytically, 

\begin{equation}
  {\mathbf z}(t) = e^{{\mathbf A}t} {\mathbf z}_0 \; ,
\end{equation}
where we have used the matrix exponential. The inverse linear 
differential equations problem is to find ${\mathbf A}$ given ${\mathbf z}(t_i)$ for
$m \geq n$ sample times $t_i$. Note that this problem is the
natural extension of the DMD to data with arbitrary sample
times.

If we assume that ${\mathbf A}$ is 
diagonalizable, we can write

\begin{equation}
  {\mathbf z}(t) = e^{{\mathbf A}t} {\mathbf z}_0 = e^{{\mathbf S}{\boldsymbol \Lambda} t {\mathbf S}^{-1}}{\mathbf z}_0 = {\mathbf S} e^{{\boldsymbol \Lambda} t} {\mathbf S}^{-1} {\mathbf z}_0 \; ,
\end{equation}
where ${\mathbf A} = {\mathbf S}{\boldsymbol \Lambda} {\mathbf S}^{-1}$ and ${\boldsymbol \Lambda}$ is diagonal. 
Let the diagonal values of ${\boldsymbol \Lambda}$ be given by 
${ \alpha}_1, \ldots , { \alpha}_k$ and define nonlinear
basis functions by $\phi_j({\boldsymbol \alpha},t) = \exp({\alpha_jt})$.
If we let ${\boldsymbol \Phi}({\boldsymbol \alpha})$ and ${\mathbf H}$ be defined as above, with 
$\eta_p(t_i) = { z}_p(t_i)$, then 

\begin{equation}
  {\mathbf H} = {\boldsymbol \Phi}({\boldsymbol \alpha}) {\mathbf B} \; ,
\end{equation}
where 

\begin{equation}
  { B}_{i,j} = { S}_{j,i} ({\mathbf S}^{-1}{\mathbf z}_0)_i \;
\end{equation}
are the entries of ${\mathbf B}$. Therefore, the inverse differential
equations problem can be solved by first solving 

\begin{equation}
  \mbox{minimize} \| {\mathbf H} - {\boldsymbol \Phi}({\boldsymbol \alpha}) {\mathbf B} \|_F
  \qquad \mbox{ over } {\boldsymbol \alpha} \in \C^n, {\mathbf B} \in \C^{n\times n} \; .
\end{equation}
Note that $k=l=n$ in this application.
The matrix ${\mathbf A}$ can then be recovered by observing that
the $i$th column of ${\mathbf B}^\intercal$ is an eigenvector of 
${\mathbf A}$ corresponding to the eigenvalue ${ \alpha}_i$.

\begin{remark}
  We note that the variable projection framework also
applies immediately to the case that $n > l$, i.e. to
the case of fitting an $l$ dimensional linear
system to data in a higher dimensional space. 
The first
algorithm presented in \cref{sec:algorithm} is
the direct result of this observation.
\end{remark}

\begin{remark} \label{rmk:confeig}
When ${\boldsymbol \alpha}$ contains confluent (or nearly confluent)
eigenvalues, the matrix ${\boldsymbol \Phi}({\boldsymbol \alpha})$ will not be 
full rank (or nearly not full rank).
For the case that ${\mathbf A}$ truly has confluent
(or nearly confluent) eigenvalues, the algorithm will suffer
near the solution. In particular, it is difficult
to try to approximate dynamics arising from a system
with a non-diagonalizable matrix ${\mathbf A}$ using exponentials
alone.
For the purposes of generalizing this method to a larger
class of ODE systems, the decomposition proposed as part of 
``Method 18'' in \cite{moler1978}
for computing the matrix exponential of a matrix ${\mathbf C}$ offers 
an interesting alternative. In Method 18, ${\mathbf C}$ is decomposed 
as ${\mathbf C} = {\mathbf S}{\mathbf B}{\mathbf S}^{-1}$, where the matrix ${\mathbf B}$ is block-diagonal, 
with each block upper-triangular. Intuitively, the blocks 
are selected so that nearly-confluent eigenvalues are grouped 
together and the condition number of ${\mathbf S}$ is kept manageable.
If we allow ${\boldsymbol \Lambda}$ to be block-diagonal with upper-triangular
blocks, then the algorithm for inverse linear
systems above could accommodate all matrices.
In this case, we may have that $k > l$ and the software will be
significantly more complicated. This is the subject of ongoing 
research and progress will be reported at a later date.
\end{remark}

\subsubsection{Modern variable projection} \label{subsec:modvarpro}

The idea at the core of variable projection,
reducing the number of unknowns in a minimization
problem by exploiting special structure,
is not limited in application to unconstrained
nonlinear least squares problems. We will not attempt
a review of this broad subject here but will
point to the applications of
\cite{oleary2013,bell2008,aravkin2012,shearer2013}
for a sense of the types of problems which can
be approached. Among the applications are
exponential data fitting with constraints,
blind deconvolution, and multiple kernel learning.
Because of this flexibility, we believe that
rephrasing the DMD as a problem in the variable
projection framework will provide opportunities
for extensions of the DMD, including constrained
and robust versions.

In the recent paper \cite{aravkin2016},
Aravkin et al. develop a variable projection method
for an interesting variation
on the exponential fitting problem.
Let ${\boldsymbol \Phi}({\boldsymbol \alpha}) \in \C^{m\times k}$ be made up of columns
of exponentials, i.e. $\phi_j({\boldsymbol \alpha},t) = \exp({\alpha_j t})$,
as in the previous section and consider a single stream
of data ${\boldsymbol \eta} \in \C^m$. The appropriate number, $k$,
of different exponentials to use to approximate the data
may be difficult to ascertain a priori.
Instead of choosing the correct number ahead of time,
one can choose a large $k$ and augment the standard
nonlinear least squares problem with a sparsity
prior, resulting in the problem

\begin{equation}
  \mbox{minimize } f({\boldsymbol \alpha},{\boldsymbol \beta}) = \|{\boldsymbol \eta} -{\boldsymbol \Phi}({\boldsymbol \alpha}){\boldsymbol \beta} \|_2^2 +
  \| {\boldsymbol \beta} \|_1 \mbox{ over } {\boldsymbol \alpha} \in \C^k, {\boldsymbol \beta} \in \C^k \; .
\end{equation}
For a fixed ${\boldsymbol \alpha}$, the problem in ${\boldsymbol \beta}$ alone
can be solved using any suitable least absolute
shrinkage and selection operator (LASSO)
algorithm.
In \cite{aravkin2016}, the function

\begin{equation}
  \tilde{f}({\boldsymbol \alpha}) = \min_{\boldsymbol \beta} f({\boldsymbol \alpha},{\boldsymbol \beta})
\end{equation}
is shown to be differentiable under suitable
conditions and a formula for the gradient is
derived. This can then be used to solve for the
minimizer of $\tilde{f}$ with a nonlinear optimization
routine.

The DMD and optimized DMD face similar issues when it comes
to determining the appropriate rank $r$.
Extending the above idea to the DMD setting is work in
progress and will be reported at a later date.

\subsection{The DMD}

In this section, we will provide some details of the
DMD algorithm and discuss its computation and
properties, using the notation and definitions
of \cite{tu2013}.

\subsubsection{Exact DMD}

The exact DMD is defined for pairs of 
data $\{ ({\mathbf x}_1,{\mathbf y}_1), \ldots , ({\mathbf x}_m,{\mathbf y}_m) \}$ which we 
assume satisfy ${\mathbf y}_j = {\mathbf A}{\mathbf x}_j$, for some matrix ${\mathbf A}$.
Typically, the pairs are assumed to be given by
equispaced snapshots of some dynamical system
${\mathbf z}(t)$, i.e. 
${\mathbf x}_j = {\mathbf z}((j-1)\Delta t)$ and ${\mathbf y}_j = {\mathbf z}(j\Delta t)$,
but they are not required to be of this form. 
For most data sets, the matrix ${\mathbf A}$ is not
determined fully by the snapshots. Therefore,
we define the matrix ${\mathbf A}$ from the data in a least-
squares sense. In particular, we set

\begin{equation} \label{eq:exactdmd}
  {\mathbf A} = {\mathbf Y}{\mathbf X}^\dagger \; ,
\end{equation}
where ${\mathbf X}^\dagger$ is the pseudo-inverse of ${\mathbf X}$.
The matrix ${\mathbf A}$ above is the 
minimizer of $\|{\mathbf A}{\mathbf X}-{\mathbf Y}\|_F$ in the case that ${\mathbf A}{\mathbf X}={\mathbf Y}$ is
over-determined and the minimum norm ($\|{\mathbf A}\|_F$) solution
of ${\mathbf A}{\mathbf X}={\mathbf Y}$ in the case that the equation is under-determined
\cite{tu2013}
($\| \cdot \|_F$ denotes the standard Frobenius norm).
We may say that ${\mathbf A}$ is the best fit linear system
mapping ${\mathbf X}$ to ${\mathbf Y}$ or, in the typical application,
the best fit linear map which advances ${\mathbf z}(t)$
to ${\mathbf z}(t+\Delta t)$ (this map is sometimes referred
to as a forward propagator).

The dynamic mode decomposition is then
defined to be the set of eigenvectors
and eigenvalues of ${\mathbf A}$. \Cref{algo:exactdmd}
provides a robust method for
computing these values \cite{tu2013}.

\begin{algorithm} 
  \caption{Exact DMD \cite{tu2013}}
  \label{algo:exactdmd}
\begin{enumerate}
  \item Define matrices ${\mathbf X}$ and ${\mathbf Y}$ from the data:
    \begin{equation} {\mathbf X} = \left ( {\mathbf x}_1, \ldots, {\mathbf x}_m \right ) \;, 
    \qquad {\mathbf Y} = \left ( {\mathbf y}_1, \ldots, {\mathbf y}_m \right ) \; . \end{equation}
  \item Take the (reduced) SVD of the matrix ${\mathbf X}$,
    i.e. compute ${\mathbf U}$, ${\boldsymbol \Sigma}$, and ${\mathbf V}$ such that
    \begin{equation} {\mathbf X} = {\mathbf U} {\boldsymbol \Sigma} {\mathbf V}^* \; ,\end{equation}
    where ${\mathbf U} \in \C^{n\times r}$, ${\boldsymbol \Sigma} \in \C^{r\times r}$,
    and ${\mathbf V} \in \C^{m\times r}$, with $r$ the rank of ${\mathbf X}$.
  \item Let $\tilde{{\mathbf A}}$ be defined by 
    \begin{equation} \label{eq:exactatilde}
      \tilde{{\mathbf A}} = {\mathbf U}^*{\mathbf Y}{\mathbf V}{\boldsymbol \Sigma}^{-1} \; .\end{equation}
  \item Compute the eigendecomposition of $\tilde{{\mathbf A}}$,
    giving a set of $r$ vectors, ${\mathbf w}$, and eigenvalues, 
    $\lambda$, such that 
    \begin{equation} \tilde{{\mathbf A}} {\mathbf w} = \lambda {\mathbf w} \; . \end{equation}
  \item For each pair $(w,\lambda)$, we have a 
    DMD eigenvalue, $\lambda$ itself, and a DMD 
    mode defined by
    \begin{equation} {\boldsymbol \varphi} = \dfrac1{\lambda} {\mathbf Y}{\mathbf V}{\boldsymbol \Sigma}^{-1} {\mathbf w} 
      \; . \label{eq:exactmode} \end{equation}
\end{enumerate}
\end{algorithm}

%
%
%
\begin{remark} \label{rmk:confeig2}
  We note that, as a mathematical matter, the matrix $\tilde{{\mathbf A}}$
  defined in \cref{eq:exactatilde}
  may not have an eigendecomposition. In this case, the Jordan 
  decomposition may be substituted for the eigendecomposition 
  and the modes which correspond to a single Jordan block
  would be considered interacting modes.
  The Jordan decomposition, however, presents severe
  numerical difficulties \cite{golub1976} and its computation
  should be avoided. 
  For a matrix without an eigendecomposition, standard eigenvalue
  decomposition algorithms will likely return 
  a result but the matrix of eigenvectors may be ill-conditioned.
  While there are some obvious alternative decompositions in
  this case, it is
  unclear which is the best alternative. The Schur decomposition
  is stable and
  would provide an orthogonal set of DMD modes but all modes 
  would interact. The block-diagonal Schur decomposition described
  in \cref{rmk:confeig} could again provide an interesting
  alternative decomposition. 
\end{remark}

\begin{remark} \label{rmk:modenorm}
  In our implementation of the exact DMD (for the computations
in \cref{sec:examples}), we use a different normalization than
that in the definition of the DMD modes \cref{eq:exactmode}. 
We set 
\begin{equation} 
{\boldsymbol \varphi} = \dfrac{1}{\|{\mathbf Y}{\mathbf V}{\boldsymbol \Sigma}^{-1} {\mathbf w} \|_2} {\mathbf Y}{\mathbf V}{\boldsymbol \Sigma}^{-1} {\mathbf w}
\end{equation}
so that the DMD modes have unit norm.
\end{remark}

\subsubsection{Low rank structure and the DMD} \label{sec:lowrank}

When computing the DMD using \cref{algo:exactdmd},
the SVD of the data is typically
truncated to avoid fitting dynamics to the lowest energy
modes, which may be corrupted by noise. The decision
of where and how to truncate can have a significant
effect on the resulting DMD modes and eigenvalues
and can vary depending on the needs of a given
application. We will focus here on so-called
\textit{hard-thresholding}, where the largest
$r$ singular values are maintained and the rest
are set to zero.

For certain applications in optimized control,
the low energy modes of a system have been found to
be important for balanced input-output models
\cite{rowley2005,rowley2006linear,rowley2006dynamics,ilak2008}.
The data in these settings typically comes from
numerical simulations, which are generally less
polluted with noise than measured data.
In this case, it may be reasonable to choose
a large $r$ for the hard threshold. 

For applications with significant measurement
error (or other sources of error), the question
of how best to truncate is difficult to answer.
Often, a heuristic choice is made, e.g. looking
for ``elbows'' in the singular value decomposition
of the data or keeping singular values up to a
certain percentage of the nuclear norm (so that
the sum of the $r$ singular values which are kept
is at least a certain percentage of the
sum of all singular values).

In the case that the measurement error is additive
white noise, the recent work of Gavish and Donoho,
\cite{gavish2014}, suggests an algorithmic choice
for the truncation.
When the standard deviation of the noise is known,
there is an analytical formula for the optimal
cut-off \cite{gavish2014}. More realistically,
the noise level must be estimated and an
alternative formula based on the median
singular value of the data is available
\cite{gavish2014}. This has the
disadvantage that at least half of the singular
values must be computed, which may be expensive for
large data sets.

Following up on the last point, when only a modest
number of singular values and vectors out of the
total are required, randomized methods for computing
the SVD can significantly reduce the cost over computing
the full SVD. Such methods are based on applying
the data matrix to a small set of random vectors
($r+p$ random vectors with $p\approx 15$ are used
to compute $r$ singular values and singular vectors)
and then computing a SVD of reduced size. Randomized
methods of this flavor are becoming increasingly
important in data analysis and dense linear algebra,
see \cite{halko2011} for a review. For an application
in the DMD setting, see \cite{erichson2016}.

\subsubsection{System reconstruction in the DMD basis}

\label{sec:reconstruction}

Let ${\mathbf z}_j = {\mathbf z}(j \Delta t)$ be snapshots of a system,
${\mathbf X} = ({\mathbf z}_0,{\mathbf z}_1, \ldots,
{\mathbf z}_m)$, and $({\boldsymbol \varphi}_i,\lambda_i)$
be $r$ DMD mode-eigenvalue pairs computed via
\cref{algo:exactdmd}, with the ${\boldsymbol \varphi}$ normalized as
in \cref{rmk:modenorm}. In many applications, it is of
interest to reconstruct the system, i.e. to compute
coefficients ${ b}_i$ so that

\begin{equation} \label{eq:reconstruct}
  {\mathbf z}_j \approx \sum_{i=1}^r { b}_i {\boldsymbol \varphi}_i \lambda_i^j \; .
\end{equation}
For example, it is then possible to extrapolate
a guess as to the future state of the system using the
formula

\begin{equation} \label{eq:extrap}
  {\mathbf z}(t) \approx \sum_{i=1}^r b_i {\boldsymbol \varphi}_i e^{\log(\lambda_i) t/\Delta t} \; .
\end{equation}

The expression \cref{eq:reconstruct} suggests the following
minimization problem for the coefficients

\begin{equation} \label{eq:mincoeffs}
  \mbox{minimize} \ \left \| {\mathbf X}
  - \begin{pmatrix}
    \vert & \vert & \\
    {\boldsymbol \varphi}_1 & {\boldsymbol \varphi}_2 & \cdots \\
    \vert & \vert & 
  \end{pmatrix} \mbox{diag}({\mathbf b})
  \begin{pmatrix}
    1 & \lambda_1 & \cdots & \lambda_1^m \\
    1 & \lambda_2 & \cdots & \lambda_2^m \\
    \vdots & \vdots & \ddots & \vdots
  \end{pmatrix} \right \|_F
  \mbox{ over } {\mathbf b} \in \C^r \; .
\end{equation}
The problem \cref{eq:mincoeffs} may be solved
with linear-algebraic methods, see \cite{jovanovic2014}
for details.
For efficiency, the coefficients may be computed
based on the first snapshot alone \cite{kutz2016}, i.e.
setting ${\mathbf b}$ as the solution of

\begin{equation}
  \mbox{minimize} \ \left \| {\mathbf z}_0
  - \begin{pmatrix}
    \vert & \vert & \\
    {\boldsymbol \varphi}_1 & {\boldsymbol \varphi}_2 & \cdots \\
    \vert & \vert & 
  \end{pmatrix} {\mathbf b} \right \|_2
  \mbox{ over } {\mathbf b} \in \C^r \; ,
\end{equation}
which can be computed using standard
linear least squares methods. In many
settings, the coefficients recovered in
this manner will be of sufficient accuracy.

The sparsity-promoting DMD method of
\cite{jovanovic2014} minimizes an objective
function similar to that of problem
\cref{eq:mincoeffs}, but augmented with a
sparsity prior, i.e. the problem

\begin{equation} \label{eq:mincoeffssparse}
  \mbox{minimize} \ \left \| {\mathbf X}
  - \begin{pmatrix}
    \vert & \vert & \\
    {\boldsymbol \varphi}_1 & {\boldsymbol \varphi}_2 & \cdots \\
    \vert & \vert & 
  \end{pmatrix} \mbox{diag}({\mathbf b})
  \begin{pmatrix}
    1 & \lambda_1 & \cdots & \lambda_1^m \\
    1 & \lambda_2 & \cdots & \lambda_2^m \\
    \vdots & \vdots & \ddots & \vdots
  \end{pmatrix} \right \|_F \, + \gamma \|{\mathbf b}\|_1
  \mbox{ over } {\mathbf b} \in \C^r \; ,
\end{equation}
where $\gamma$ is a parameter chosen to
control the number of nonzero terms in ${\mathbf b}$.
This results in a parsimonious representation.

If the goal is the best reconstruction possible
for the given time dynamics,
then it seems that the DMD modes as returned
by the exact DMD may be ignored. A high-quality
reconstruction is given by

\begin{equation} \label{eq:extrap2}
  {\mathbf z}(t) \approx \sum_{i=1}^r {\boldsymbol \psi}_i e^{\log(\lambda_i) t/\Delta t} \; ,
\end{equation}
where the ${\boldsymbol \psi}_i$ solve 

\begin{equation} \label{eq:mincoeffs2}
  \mbox{minimize} \ \left \| {\mathbf X}
  - \begin{pmatrix}
    \vert & \vert & \\
    {\boldsymbol \psi}_1 & {\boldsymbol \psi}_2 & \cdots \\
    \vert & \vert & 
  \end{pmatrix}
  \begin{pmatrix}
    1 & \lambda_1 & \cdots & \lambda_1^m \\
    1 & \lambda_2 & \cdots & \lambda_2^m \\
    \vdots & \vdots & \ddots & \vdots
  \end{pmatrix} \right \|_F
  \mbox{ over } {\boldsymbol \psi}_1,\ldots,{\boldsymbol \psi}_r \in \C^n \; 
\end{equation}
and are computable with standard linear least squares
methods. This does not result in a parsimonious
representation, as in \cite{jovanovic2014}, but the
fit will be optimal.

\begin{remark}
  Upon examining equation \cref{eq:extrap}, it is clear
that the dynamic mode decomposition is purely a method of 
fitting exponentials to data. This is equivalent to recovering 
the eigendecomposition of the underlying linear system in the 
case that that linear system is diagonalizable. In the case
that there are transient dynamics which are not captured by a purely
diagonal system, the extensions mentioned in
\cref{rmk:confeig,rmk:confeig2}
provide an alternative. Of course,
exponentials with similar exponents can mimic terms of
the form $t\exp( \alpha t)$, but there may be lots of cancellation
in the intermediate calculations.
\end{remark}

\subsubsection{Bias of the DMD} \label{subsec:bias}

The papers \cite{hemati2015,dawson2016} deal with the
question of bias in the computed DMD modes and
eigenvalues when data is corrupted by sensor noise.
In the case of additive white noise in the
measurements, there are formulas for the bias
associated with the exact DMD algorithm
\cite{dawson2016}. If $m$ is the number of snapshots
and $n$ is the dimension of the system, the bias
will be the dominant component of the DMD error whenever
$\sqrt{m} R_{STN} > \sqrt{n}$, where $R_{STN}$
is the signal-to-noise ratio. For the purpose
of avoiding this pitfall, there are
a number of alternative, debiased algorithms.
We'll present two of them here: the fbDMD (forward-backward
DMD) and tlsDMD (total least-squares DMD).

Let ${\mathbf X}$ and ${\mathbf Y}$ be as in the exact DMD and let
${\mathbf U}_{ X}{\boldsymbol \Sigma}_{ X} {\mathbf V}^*_{ X} = {\mathbf X}$ and ${\mathbf U}_{ Y} {\boldsymbol \Sigma}_{ Y} {\mathbf V}^*_{ Y} = {\mathbf Y}$
be (reduced) SVDs of these matrices. The debiased
algorithms will follow the steps of the exact DMD
and will only differ in the definition of $\tilde{{\mathbf A}}$.

Intuitively, the fbDMD method can be thought of as
a correction to the one-directional preference of the
exact DMD. We define two matrices

\begin{equation}
  \tilde{{\mathbf A}}_f = {\mathbf U}_{ X}^* {\mathbf Y} {\mathbf V}_{ X} {\boldsymbol \Sigma}_{ X}^{-1}
\end{equation}
and

\begin{equation}
  \tilde{{\mathbf A}}_b = {\mathbf U}_{ Y}^* {\mathbf X} {\mathbf V}_{ Y} {\boldsymbol \Sigma}_{ Y}^{-1}
\end{equation}
which represent forward and backward propagators
for the data in the same manner as the exact DMD.
The matrix given by

\begin{equation}
  \tilde{{\mathbf A}} = \left (\tilde{{\mathbf A}}_f \tilde{{\mathbf A}}_b^{-1} \right)^{1/2}
\end{equation}
is then a debiased estimate of the forward
propagator \cite{dawson2016}.

Because of the
nonuniqueness of the square root, some care
must be taken in the calculation of $\tilde{{\mathbf A}}$.
In particular, the eigenvalues of $\tilde{{\mathbf A}}$
are only determined up to a factor of $\pm 1$
by the square root. Dawson et al. recommend
choosing the square root which is closest
to $\tilde{{\mathbf A}}_f$ in norm. Na\"{i}vely this is
a $\bigo(2^r)$ calculation, where $r$ is the
number of eigenvalues, and it is unclear how
to improve on this scaling.

The tlsDMD method attempts to correct for the
fact that noise on ${\mathbf X}$ and noise on ${\mathbf Y}$ are not
treated in the same way by the exact DMD. First, we
project ${\mathbf X}$ and ${\mathbf Y}$ onto $r < m/2$ POD modes,
obtaining $\tilde{{\mathbf X}}$ and $\tilde{{\mathbf Y}}$. We define

\begin{equation}
  {\mathbf Z} = \begin{pmatrix} \tilde{{\mathbf X}} \\ \tilde{{\mathbf Y}} 
  \end{pmatrix} \; 
\end{equation}
and compute its (reduced) SVD, ${\mathbf U}_{ Z} {\boldsymbol \Sigma}_{ Z} {\mathbf V}_{ Z}^* = {\mathbf Z}$.
Let ${\mathbf U}_{11} = {\mathbf U}_{ Z}(1 :  r,1 :  r)$ and
${\mathbf U}_{21} = {\mathbf U}_{ Z}(r+1 :  2r,1 :  r)$.
Then the matrix given by

\begin{equation}
  \tilde{{\mathbf A}} = {\mathbf U}_{21}{\mathbf U}_{11}^{-1}
\end{equation}
is a debiased estimate of the forward
propagator \cite{dawson2016}. This definition is distinct
from but similar in spirit to the definition
of \cite{hemati2015}.

\begin{remark}
  We note that, as described above, the
  forward-backward DMD is somewhat approximate
  in that $\tilde{{\mathbf A}}_f$ and $\tilde{{\mathbf A}}_b$ are
  representations of the underlying operator
  which are projected onto different subspaces.
  Let ${\mathbf A}$ denote the underlying linear
  operator and ${\mathbf S}_f = {\mathbf Y}{\mathbf V}_{ X}{\boldsymbol \Sigma}_{ X}^{-1}$
  and ${\mathbf S}_b = {\mathbf X}{\mathbf V}_{ Y}{\boldsymbol \Sigma}_{ Y}^{-1}$. Then ${\mathbf A}_f  = {\mathbf S}_f \tilde{{\mathbf A}}_f
  {\mathbf S}_f^\dagger$ and ${\mathbf A}_b = {\mathbf S}_b \tilde{{\mathbf A}}_b {\mathbf S}_b^\dagger$
  are both approximations of ${\mathbf A}$, but
  $\tilde{{\mathbf A}}_f$ and $\tilde{{\mathbf A}}_b$ are not
  necessarily good approximations of each
  other. We recommend instead computing
  the full approximations to ${\mathbf A}_f$ and ${\mathbf A}_b$
  for the data projected onto the first
  $r$ POD modes. And using the approximation
  ${\mathbf A}_{fb} = ({\mathbf A}_f {\mathbf A}_b^{-1})^{1/2}$ for the
  linear operator projected onto those modes.
  For the calculations in \cref{sec:examples},
  we used this modification to the fbDMD.
\end{remark}

In \cref{sec:examples}, we compare the
behavior of these methods with the optimized DMD
for data with sensor noise.

\section{The optimized DMD} \label{sec:algorithm}

In this section, we will combine ideas from variable
projection and the DMD literature to obtain a pair of 
debiased algorithms which can compute the DMD for 
data with arbitrary sample times. We will then 
demonstrate some of the properties of this algorithm 
in the following section.

\subsection{Algorithm}

Let ${\mathbf X} = ({\mathbf z}_0,\ldots,{\mathbf z}_m)$ be a matrix of snapshots,
with ${\mathbf z}_j = {\mathbf z}(t_j) \in \C^n$ for a set of times $t_j$. 
For a target rank $r$ determined by the user, assume
that the data is the solution of a linear system of
differential equations, restricted to a subspace
of dimension $r$. I.e., assume that 

\begin{equation} \label{eq:rankrlinsys}
  {\mathbf z}(t) \approx {\mathbf S} e^{{\boldsymbol \Lambda} t} {\mathbf S}^\dagger {\mathbf z}_0 \; ,
\end{equation}
where ${\mathbf S} \in \C^{n\times r}$ and ${\boldsymbol \Lambda} \in \C^{r\times r}$.
As in \cref{sec:invdiffeq}, we may rewrite
\cref{eq:rankrlinsys} as

\begin{equation}
{\mathbf X}^\intercal \approx {\boldsymbol \Phi}({\boldsymbol \alpha}) {\mathbf B} \; ,
\end{equation}
where

\begin{equation}
{ B}_{i,j} = { S}_{j,i} \left ({\mathbf S}^\dagger {\mathbf z}_0 \right)_i
\end{equation}
and ${\boldsymbol \Phi}({\boldsymbol \alpha}) \in \C^{(m+1)\times r}$ with entries defined
by ${ \Phi}({\boldsymbol \alpha})_{i,j} = \exp (\alpha_j t_i)$.

The preceding leads us to the following definition
of the optimized DMD in terms of an exponential fitting
problem. Suppose that
$\hat{{\boldsymbol \alpha}}$ and $\hat{{\mathbf B}}$ solve 

\begin{equation} \label{eq:optdmdprob}
  \mbox{minimize} \| {\mathbf X}^\intercal - {\boldsymbol \Phi}({\boldsymbol \alpha}) {\mathbf B} \|_F
  \qquad \mbox{ over } {\boldsymbol \alpha} \in \C^k, {\mathbf B} \in \C^{l\times n} \; .
\end{equation}
The optimized DMD eigenvalues are then defined by
$\lambda_i = \hat{{ \alpha}}_i$ and the eigenmodes are
defined by 

\begin{equation}
  {\boldsymbol \varphi}_i = \dfrac{1}{\|\hat{{\mathbf B}}^\intercal( : ,i) \|_2}
  \hat{{\mathbf B}}^{\intercal}( : ,i) \; ,
\end{equation}
where $\hat{{\mathbf B}}^\intercal( : ,i)$ is the $i$-th column
of $\hat{{\mathbf B}}^\intercal$. We summarize the above definition
as \cref{algo:optdmd}; note that this definition
of the optimized DMD is morally equivalent to that
presented in \cite{chen2012}.

If we set $b_i = \|\hat{{\mathbf B}}^\intercal( : ,i)\|_2$, then 

\begin{equation} \label{eq:reconstructu}
  \tilde{{\mathbf z}_j} = \sum_{i=1}^r b_i e^{\lambda_i t_j} {\boldsymbol \varphi}_i
\end{equation}
is an approximation to ${\mathbf z}_j$ for each $j = 0, \ldots, m$. 
Therefore, if $\hat{{\mathbf B}}$ and $\hat{{\boldsymbol \alpha}}$ are computed,
the question of system reconstruction in the optimized DMD basis
is trivial.

\begin{algorithm} 
  \caption{Optimized DMD}
  \label{algo:optdmd}
  \begin{enumerate}
  \item Let the snapshot matrix ${\mathbf X}$ and an initial guess for 
    ${\boldsymbol \alpha}$ be given.
  \item Solve the problem
    \begin{equation} \label{eq:algoprobfull}
      \mbox{minimize} \| {\mathbf X}^\intercal - {\boldsymbol \Phi}({\boldsymbol \alpha}) {\mathbf B} \|_F
      \qquad \mbox{ over } {\boldsymbol \alpha} \in \C^k, {\mathbf B} \in \C^{l\times n} \; ,
    \end{equation}
    using a variable projection algorithm.
  \item Set $\lambda_i = \hat{{ \alpha}}_i$ and
    \begin{equation}
      {\boldsymbol \varphi}_i = \dfrac{1}{\|\hat{{\mathbf B}}^\intercal( : ,i) \|_2}
      \hat{{\mathbf B}}^{\intercal}( : ,i) \; ,    
    \end{equation}
    saving the values $b_i = \| \hat{{\mathbf B}}^\intercal( : ,i)\|_2$.
  \end{enumerate}
\end{algorithm}

\begin{remark}
  We note that, given a single initial guess for 
  ${\boldsymbol \alpha}$, the Levenberg-Marquardt
  algorithm will not necessarily converge to the 
  global minimizer of \cref{eq:optdmdprob}. It is
  therefore technically incorrect to claim that 
  \cref{algo:optdmd} will always compute the optimized
  DMD of a given set of snapshots. Indeed, it may be
  that the proper way to view \cref{algo:optdmd,algo:approxoptdmd} 
  is as post-processors for the initial guess for 
  ${\boldsymbol \alpha}$, which improve on 
  ${\boldsymbol \alpha}$ by computing a nearby local
  minimizer. In \cref{sec:examples}, we
  see that this post-processing --- even when it is
  unclear whether we've computed the global minimizer ---
  provides significant improvement over other DMD
  methods.
\end{remark}

The asymptotic cost of \cref{algo:optdmd} can be
estimated using the formulas from
\cref{subsec:mrhs}. For each iteration of the variable
projection algorithm, the cost is $\bigo (r^2 mn)$.
For large $m$ and $n$, it is possible to compute the
optimized DMD (or an approximation of the optimized DMD)
more efficiently. Suppose that instead of computing
$\hat{{\boldsymbol \alpha}}$ and $\hat{{\mathbf B}}$ which solve \cref{eq:optdmdprob},
you computed $\breve{{\boldsymbol \alpha}}$ and $\breve{{\mathbf B}}$ which solve
\begin{equation} \label{eq:optdmdprobtrunc}
  \mbox{minimize} \| {\mathbf X}_r^\intercal - {\boldsymbol \Phi}({\boldsymbol \alpha}) {\mathbf B} \|_F
  \qquad \mbox{ over } {\boldsymbol \alpha} \in \C^k, {\mathbf B} \in \C^{l\times n} \; ,
\end{equation}
where ${\mathbf X}_r$ is the optimal rank $r$ approximation
of ${\mathbf X}$ (in the Frobenius norm). \Cref{algo:approxoptdmd}
computes the solution to this problem.

\begin{algorithm}
\caption{Approximate optimized DMD}
  \label{algo:approxoptdmd}
  \begin{enumerate}
  \item Let the snapshot matrix ${\mathbf X}$ and an initial
    guess for ${\boldsymbol \alpha}$ be given.
  \item Compute the truncated SVD of ${\mathbf X}$ of
    rank $r$, i.e. compute ${\mathbf U}_r \in \C^{n\times r}$
    ${\boldsymbol \Sigma}_r \in \C^{r\times r}$, and
    ${\mathbf V}_r \in \C^{(m+1)\times r}$ such that
    \begin{equation}
      {\mathbf X}_r = {\mathbf U}_r{\boldsymbol \Sigma}_r {\mathbf V}_r^* \; .
    \end{equation}
  \item Compute $\grave{{\boldsymbol \alpha}}$ and $\grave{{\mathbf B}}$ which solve
    \begin{equation} \label{eq:algoprobtrunc}
  \mbox{minimize} \| \conj{{\mathbf V}_r}{\boldsymbol \Sigma} - {\boldsymbol \Phi}({\boldsymbol \alpha}) {\mathbf B} \|_F
  \qquad \mbox{ over } {\boldsymbol \alpha} \in \C^r, {\mathbf B} \in \C^{r\times r} \; ,
    \end{equation}
    using a variable projection algorithm.
  \item Set $\lambda_i = \grave{{ \alpha}}_i$ and

    \begin{equation}
      {\boldsymbol \varphi}_i = \dfrac1{\|{\mathbf U}_r\grave{{\mathbf B}}^\intercal( : ,i)\|_2}
      {\mathbf U}_r\grave{{\mathbf B}}^\intercal( : ,i)  \; ,
    \end{equation}
    saving the values $b_i = \|{\mathbf U}_r\grave{{\mathbf B}}^\intercal( : ,i)\|_2$.
  \end{enumerate}
\end{algorithm}

The cost of computing the rank $r$ SVD in
step 2 of \cref{algo:approxoptdmd} is
$\bigo (mn\min (m,n))$ using a standard algorithm
or $\bigo (r^2(m+n)+rmn)$ using a randomized algorithm
\cite{halko2011} (the constant is larger for the
randomized algorithm, so determining the faster
algorithm can be subtle). After this is computed
once, the cost for each step of the variable
projection algorithm is improved to $\bigo(r^3m)$.
This can lead to significant speed ups over the
original.

The following proposition shows the relation between the
minimization problem \cref{eq:optdmdprobtrunc} and 
\cref{algo:approxoptdmd}.

\begin{proposition}
  Let ${\mathbf U}_r$, ${\boldsymbol \Sigma}_r$, ${\mathbf V}_r$, $\grave{{\boldsymbol \alpha}}$, and 
$\grave{{\mathbf B}}$ be as in \cref{algo:approxoptdmd}. Then 
$\breve{{\boldsymbol \alpha}} = \grave{{\boldsymbol \alpha}}$ and 
$\breve{{\mathbf B}} = \grave{{\mathbf B}} {\mathbf U}_r^\intercal$ are solutions of 
\cref{eq:optdmdprobtrunc}. 
\end{proposition}

\begin{proof}
  Let ${\mathbf U}$ and ${\mathbf V}$ be orthogonal matrices such that their first $r$ 
columns equal ${\mathbf U}_r$ and ${\mathbf V}_r$ respectively. We note that 
\begin{align}
  \| \conj{{\mathbf V}_r} {\boldsymbol \Sigma}_r - {\boldsymbol \Phi}(\grave{{\boldsymbol \alpha}}) \grave{{\mathbf B}} \|_F
&= \| \conj{{\mathbf V}_r} \begin{bmatrix} {\boldsymbol \Sigma}_r & 0 \\ 0 & 0 
  \end{bmatrix} - {\boldsymbol \Phi}(\grave{{\boldsymbol \alpha}}) \begin{bmatrix} \grave{{\mathbf B}}
& 0 
  \end{bmatrix} \|_F \nonumber \\
&=  \| \conj{{\mathbf V}} \begin{bmatrix} {\boldsymbol \Sigma}_r & 0 \\ 0 & 0 
  \end{bmatrix} {\mathbf U}^\intercal - {\boldsymbol \Phi}(\grave{{\boldsymbol \alpha}}) 
\grave{{\mathbf B}} \begin{bmatrix} {\mathbf I}_r & 0 \end{bmatrix} {\mathbf U}^\intercal \|_F \nonumber \\
&= \| {\mathbf X}_r^\intercal - {\boldsymbol \Phi}(\breve{{\boldsymbol \alpha}}) \breve{{\mathbf B}} \|_F \; .
\end{align}

By contradiction, assume that there exist $\dot{{\boldsymbol \alpha}}$ and
$\dot{{\mathbf B}}$ such that 

\begin{equation}
  \| {\mathbf X}_r^\intercal - {\boldsymbol \Phi}(\dot{{\boldsymbol \alpha}}) \dot{{\mathbf B}}\|_F < \|{\mathbf X}_r^\intercal - 
{\boldsymbol \Phi}(\breve{{\boldsymbol \alpha}}) \breve{{\mathbf B}} \|_F \; .
\end{equation}
Then,

\begin{align}
\| \conj{{\mathbf V}_r} {\boldsymbol \Sigma}_r - {\boldsymbol \Phi}(\grave{{\boldsymbol \alpha}}) \grave{{\mathbf B}} \|_F
&> \| {\mathbf X}_r^\intercal - {\boldsymbol \Phi}(\dot{{\boldsymbol \alpha}}) \dot{{\mathbf B}}\| \nonumber \\
&= \| \conj{{\mathbf V}} \begin{bmatrix} {\boldsymbol \Sigma}_r & 0 \\ 0 & 0 
  \end{bmatrix} {\mathbf U}^\intercal - {\boldsymbol \Phi}(\dot{{\boldsymbol \alpha}}) \dot{{\mathbf B}} \|_F \nonumber \\
&= \| \conj{{\mathbf V}} \begin{bmatrix} {\boldsymbol \Sigma}_r & 0 \\ 0 & 0 
  \end{bmatrix} - {\boldsymbol \Phi}(\dot{{\boldsymbol \alpha}}) \dot{{\mathbf B}} \conj{{\mathbf U}} \|_F 
\nonumber \\
&\geq \| \conj{{\mathbf V}} \begin{bmatrix} {\boldsymbol \Sigma}_r & 0 \\ 0 & 0 
  \end{bmatrix} - {\boldsymbol \Phi}(\dot{{\boldsymbol \alpha}}) \begin{bmatrix} \dot{{\mathbf B}} \conj{{\mathbf U}_r} &
0 \end{bmatrix} \|_F \nonumber \\
&= \| \conj{{\mathbf V}_r} {\boldsymbol \Sigma}_r - {\boldsymbol \Phi}(\dot{{\boldsymbol \alpha}}) \dot{{\mathbf B}} \conj{{\mathbf U}_r} \|_F .
\end{align}
By the definition of $\grave{{\boldsymbol \alpha}}$ and $\grave{{\mathbf B}}$, we have 
that 

\begin{equation}
\| \conj{{\mathbf V}_r} {\boldsymbol \Sigma}_r - {\boldsymbol \Phi}(\dot{{\boldsymbol \alpha}}) \dot{{\mathbf B}} \conj{{\mathbf U}_r} \|_F
\geq \| \conj{{\mathbf V}_r} {\boldsymbol \Sigma}_r - {\boldsymbol \Phi}(\grave{{\boldsymbol \alpha}}) \grave{{\mathbf B}} \|_F \; ,
\end{equation}
a contradiction.
\end{proof}

The solution of the minimization problem \cref{eq:optdmdprobtrunc} 
is desirable in and of 
itself in many situations. In particular, consider the cases in 
which you replace the data ${\mathbf X}$ with ${\mathbf X}_r$ for the purpose of 
denoising the data or restricting the data to some low-dimensional 
structure (see \cref{sec:lowrank} for more). In the general 
case, the following
proposition shows that the eigenvalues and eigenmodes 
produced by \cref{algo:optdmd,algo:approxoptdmd}
will often give reconstructions of 
comparable quality. 

\begin{proposition}
  Suppose that the pair $(\breve{{\boldsymbol \alpha}},\breve{{\mathbf B}})$ is a solution
of \cref{eq:optdmdprobtrunc} and that $(\hat{{\boldsymbol \alpha}},\hat{{\mathbf B}})$
is a solution of \cref{eq:optdmdprob}. Then

\begin{equation}
  \| {\mathbf X}^\intercal - {\boldsymbol \Phi}(\breve{{\boldsymbol \alpha}}) \breve{{\mathbf B}} \|_F
\leq 2 \| {\mathbf X}^\intercal - {\mathbf X}_r^\intercal \|_F 
+ \| {\mathbf X}^\intercal - {\boldsymbol \Phi}(\hat{{\boldsymbol \alpha}})\hat{{\mathbf B}} \|_F
\leq 3 \| {\mathbf X}^\intercal - {\boldsymbol \Phi}(\hat{{\boldsymbol \alpha}})\hat{{\mathbf B}} \|_F \; .
\end{equation}

\end{proposition}

\begin{proof}
  Using the definitions of $(\breve{{\boldsymbol \alpha}},\breve{{\mathbf B}})$
and $(\hat{{\boldsymbol \alpha}},\hat{{\mathbf B}})$, we have 

\begin{align}
  \| {\mathbf X}^\intercal - {\boldsymbol \Phi}(\breve{{\boldsymbol \alpha}}) \breve{{\mathbf B}} \|_F 
&\leq \| {\mathbf X}^\intercal - {\mathbf X}_r^\intercal \|_F + \| {\mathbf X}_r^\intercal 
- {\boldsymbol \Phi}(\breve{{\boldsymbol \alpha}}) \breve{{\mathbf B}} \|_F  \nonumber \\
&\leq \|{\mathbf X}^\intercal - {\mathbf X}_r^\intercal\|_F + \| {\mathbf X}_r^\intercal 
- {\boldsymbol \Phi}(\hat{{\boldsymbol \alpha}}) \hat{{\mathbf B}} \|_F  \nonumber \\
&\leq 2\|{\mathbf X}^\intercal - {\mathbf X}_r^\intercal\|_F + \| {\mathbf X}^\intercal 
- {\boldsymbol \Phi}(\hat{{\boldsymbol \alpha}}) \hat{{\mathbf B}} \|_F  \nonumber \\
&\leq 3 \| {\mathbf X}^\intercal 
- {\boldsymbol \Phi}(\hat{{\boldsymbol \alpha}}) \hat{{\mathbf B}} \|_F \; ,
\end{align}
as desired.
\end{proof}

\subsection{Initialization}

For good performance, the Levenberg-Marquardt algorithm,
which is at the heart of the variable projection method used to 
solve problems \cref{eq:algoprobfull} and 
\cref{eq:algoprobtrunc}, requires a good initial guess
for the parameters ${\boldsymbol \alpha}$. Let the data ${\mathbf X} = ({\mathbf z}_0,\ldots, {\mathbf z}_m)$ be as in 
the previous subsection, with ${\mathbf z}_j = {\mathbf z}(t_j) \in \C^n$. 
We will assume here that the sample times $t_j$ are in increasing order, 
$t_0 < t_1 < \cdots < t_m$. We propose using a finite
difference style approximation to obtain an initial guess.

We assume arguendo that the ${\mathbf z}_j$ are iterates of a 
finite difference scheme, with timesteps $t_j$, applied to
the ODE system

\begin{equation}
  \dot{{\mathbf z}} = {\mathbf A}{\mathbf z} \; .
\end{equation}
If the ${\mathbf z}_j$ were obtained using the trapezoidal rule, 
then we have 

\begin{equation}
  \dfrac{{\mathbf z}_j-{\mathbf z}_{j-1}}{t_j-t_{j-1}} = \dfrac12 {\mathbf A} ({\mathbf z}_j+{\mathbf z}_{j-1}) \; ,
\end{equation}
for $j = 1,\ldots, m$. Let ${\mathbf X}_1 = ({\mathbf z}_0,\ldots,{\mathbf z}_{m-1})$, 
${\mathbf X}_2 = ({\mathbf z}_1,\ldots,{\mathbf z}_m)$, and $T = \mbox{diag}(t_1-t_0,t_2-t_1,\ldots,t_m-t_{m-1})$. 
Then ${\mathbf A}$ satisfies 

\begin{equation}
  ({\mathbf X}_2-{\mathbf X}_1)T^{-1} = {\mathbf A} \dfrac{{\mathbf X}_1+{\mathbf X}_2}{2} \; .
\end{equation}

We can then use an exact DMD-like algorithm to 
approximate the eigenvalues of ${\mathbf A}$ (which are then
our initial guess for ${\boldsymbol \alpha}$). \Cref{algo:initialize}
takes as input ${\mathbf X} = ({\mathbf z}_0,\ldots,{\mathbf z}_m)$ and $t_0 < \cdots < t_m$ 
and outputs approximate eigenvalues ($\lambda_i$) of ${\mathbf A}$.

\begin{algorithm}
  \caption{Initialization routine}
  \label{algo:initialize}
\begin{enumerate}
  \item Let ${\mathbf X}_1$, ${\mathbf X}_2$, and $T$ be as above.  
Define matrices ${\mathbf Y}$ and ${\mathbf Z}$ from the data:
    \begin{equation} {\mathbf Y} = \dfrac{{\mathbf X}_1+{\mathbf X}_2}2 , 
    \qquad {\mathbf Z} = \left ( {\mathbf X}_2 - {\mathbf X}_1 \right ) T^{-1} \; . \end{equation}
  \item Take the (reduced) SVD of the matrix ${\mathbf Y}$,
    i.e. compute ${\mathbf U}$, ${\boldsymbol \Sigma}$, and ${\mathbf V}$ such that
    \begin{equation} {\mathbf Y} = {\mathbf U} {\boldsymbol \Sigma} {\mathbf V}^* \; ,\end{equation}
    where ${\mathbf U} \in \C^{n\times r}$, ${\boldsymbol \Sigma} \in \C^{r\times r}$,
    and ${\mathbf V} \in \C^{m\times r}$, with $r$ the rank of ${\mathbf Y}$.
  \item Let $\tilde{{\mathbf A}}$ be defined by 
    \begin{equation} \tilde{{\mathbf A}} = {\mathbf U}^*{\mathbf Z}{\mathbf V}{\boldsymbol \Sigma}^{-1} \; .\end{equation}
  \item Compute the eigendecomposition of $\tilde{{\mathbf A}}$,
    giving a set of $r$ vectors, ${\mathbf w}$, and eigenvalues, 
    $\lambda$, such that 
    \begin{equation} \tilde{{\mathbf A}} {\mathbf w} =
      \lambda {\mathbf w} \; . \end{equation}
  \item Return the eigenvalues. 
\end{enumerate}
\end{algorithm}

In the above discussion, we chose the trapezoidal rule
but any number of finite difference schemes are 
applicable (in particular, the Adams-Bashforth/Adams-Moulton
family of discretization schemes). We opted for the
trapezoidal rule because it treats the data symmetrically
and has favorable stability properties for 
oscillatory phenomena. It is certainly possible that 
a different choice of finite difference scheme would
be more appropriate, depending on the desired 
accuracy and stability properties for the given
application. 

We also note that if a solution of \cref{eq:optdmdprob}
is required for a certain application, then the solution
to \cref{eq:optdmdprobtrunc} can provide a good 
initial condition.

Finally, for equispaced $t_j$, the eigenvalues
returned by the exact DMD (or one of the debiased
methods of \cref{subsec:bias}) can serve
as an initial guess, after taking the logarithm and
scaling appropriately. In this sense, the optimized DMD
can be viewed as a post-processing step for the
original DMD algorithm.

\section{Examples} \label{sec:examples}

In this section, we present some numerical results 
in order to discuss the performance of the 
optimized DMD, as computed using the tools 
discussed above.
All calculations were performed in MATLAB on a 
laptop with an Intel Core i7-6600U CPU and 
16Gb of memory. The code used to generate these 
figures is available online \cite{askham2017figcode}.

For these examples, we use two notions of
system reconstruction to evaluate the
quality of the DMDs we compute. Suppose that
${\mathbf X}$ is our matrix of snapshots and ${\mathbf A}$ is the
true underlying system matrix. If we compute
$r$ DMD eigenvalues $\lambda_i$
and modes ${\boldsymbol \varphi}_i$, then an approximation
of ${\mathbf A}$ may be recovered via

\begin{equation}
  {\mathbf A} \approx ({\boldsymbol \varphi}_1 \cdots {\boldsymbol \varphi}_r ) \mbox{diag}(\lambda_1, \ldots, \lambda_r)
  ({\boldsymbol \varphi}_1 \cdots {\boldsymbol \varphi}_r )^\dagger \; .
\end{equation}
In the case that ${\mathbf A}$ is known, we can compare
this reconstruction with ${\mathbf A}$ in the Frobenius norm.

We can also consider the reconstruction of
${\mathbf X}$ given by our decomposition. As noted
in \cref{sec:reconstruction}, the quality of this
reconstruction can depend on the definition
used for determining the coefficients. In
order to put all methods on a level playing field,
we will take the reconstruction of the snapshots
to be the projection of the data onto the
time dynamics given by the computed eigenvalues.
Let ${\boldsymbol \Phi}({\boldsymbol \alpha})$ be the matrix of exponentials
as used in the definition of the optimized DMD.
Then we will define the reconstruction of the
snapshots to be $\left ({\boldsymbol \Phi}({\boldsymbol \alpha}){\boldsymbol \Phi}({\boldsymbol \alpha})^\dagger {\mathbf X}^\intercal \right )^\intercal$.
A measure of the reconstruction quality is then given
by the relative Frobenius norm of the residual, i.e.

\begin{equation} \label{eq:optrecerr}
  \dfrac{ \| {\mathbf X}^\intercal - {\boldsymbol \Phi}({\boldsymbol \alpha}){\boldsymbol \Phi}({\boldsymbol \alpha})^\dagger {\mathbf X}^\intercal \|_F}
  {\| {\mathbf X}^\intercal \|_F} \;.
\end{equation}

\subsection{Synthetic data}

First, we revisit some of the synthetic
data examples of \cite{dawson2016} in order 
to discuss the effect of noise on the optimized DMD.

\subsubsection{Example 1: measurement noise, periodic system}

Let ${\mathbf z}(t)$ be the solution of a two dimensional
linear system with the following dynamics

\begin{equation} \label{eq:example1}
  \ddot{{\mathbf z}} = \begin{pmatrix} 1 & -2 \\ 1 & -1 
  \end{pmatrix} {\mathbf z} \; .
\end{equation}
We use the initial condition ${\mathbf z}(0) = (1,0.1)^\intercal$
and take snapshots ${\mathbf z}_j = {\mathbf z}(j \Delta t) + \sigma {\mathbf g}$ with
$\Delta t = 0.1$, $\sigma$ a prescribed noise level, and
${\mathbf g}$ a vector whose entries are drawn from a standard
normal distribution.
The continuous time eigenvalues of this system
are $\pm i$ (this is how the optimized DMD computes
eigenvalues) and the discrete time eigenvalues are
$\exp(\pm \Delta t i)$. These dynamics should display
neither growth nor decay, but in the presence of noise,
the exact DMD eigenvalues have a negative
real part because of inherent bias in the definition
\cite{dawson2016}.

We consider the effect of both the size of the noise,
$\sigma$, and the number of snapshots, $m$, on the
quality of the modes and eigenvalues obtained
from various methods. We set the noise level to the
values $\sigma^2 = 10^{-1}, 10^{-3}, \ldots, 10^{-9}$
and run tests with $m = 2^6, 2^7, \ldots, 2^{13}$
snapshots. 
For each noise level and number of snapshots,
we compute the eigenvalues and modes of this system
using the exact DMD, fbDMD, tlsDMD, and optimized
DMD over 1000 trials (different draws of the vector
${\mathbf g}$). 

\begin{figure}[h!]

\centering

\includegraphics[width=.8\textwidth]{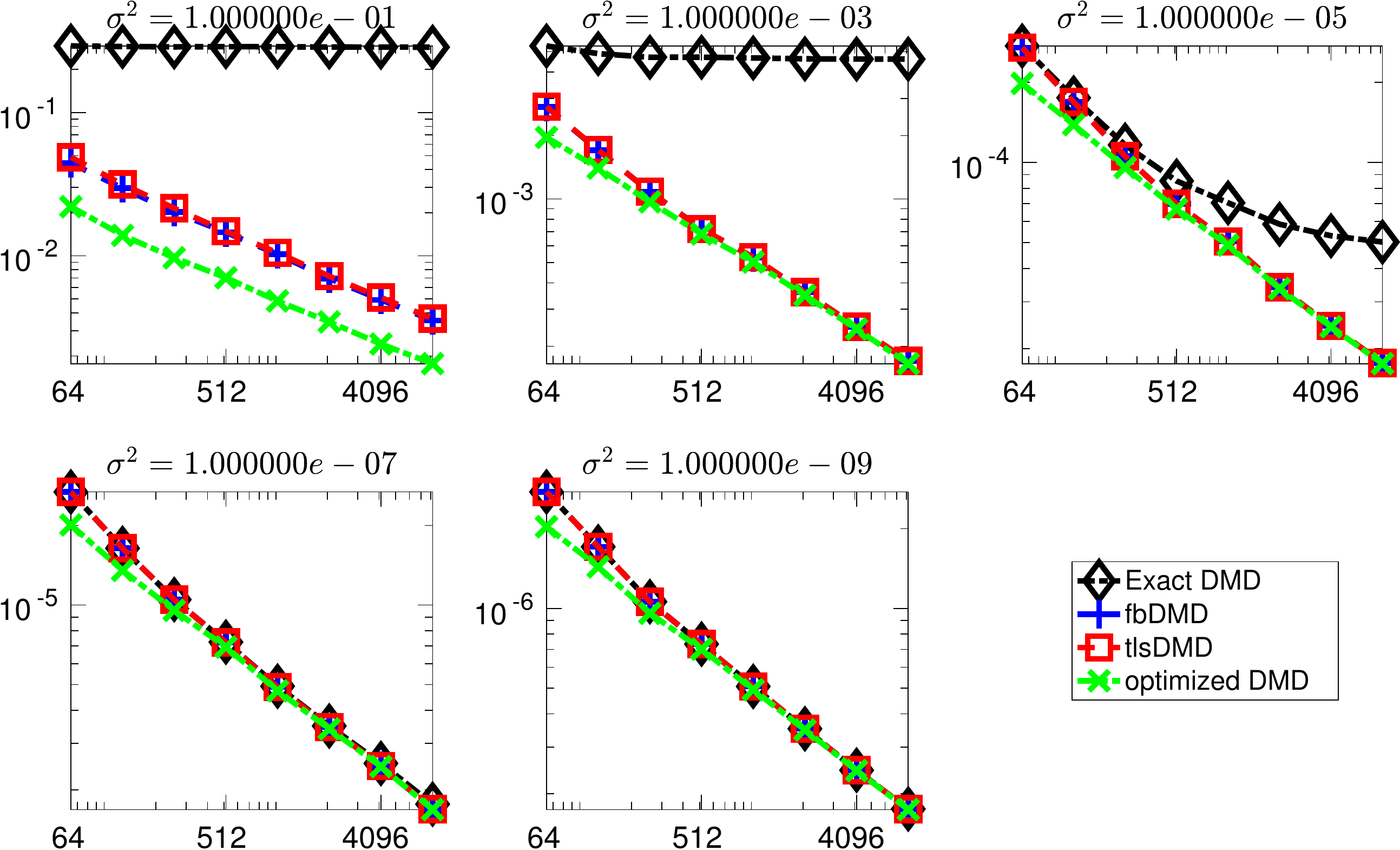}

\caption{Example 1. This figure shows the mean Frobenius norm error 
(averaged over 1000 runs) in the reconstructed
system matrix ${\mathbf A}$ as a function of the number of snapshots
$m$ for various noise levels $\sigma^2$.}

\label{fig:senper_reca}

\end{figure}

In \cref{fig:senper_reca}, we plot the mean Frobenius
norm error in the reconstructed system matrix (averaged
over the trials) as a function of the
number of snapshots for various noise levels.
We see that, as in \cite{dawson2016}, the error in the
exact DMD eventually levels off at the higher noise
levels because of the bias in its eigenvalues. The
other methods perform well, with the error decaying
as the number of snapshots increases. The optimized DMD
is shown to have an advantage over the fbDMD and tlsDMD
in this measure
primarily at the highest noise levels and with the
fewest snapshots. 

\begin{figure}[h!]

\centering

\includegraphics[width=.8\textwidth]{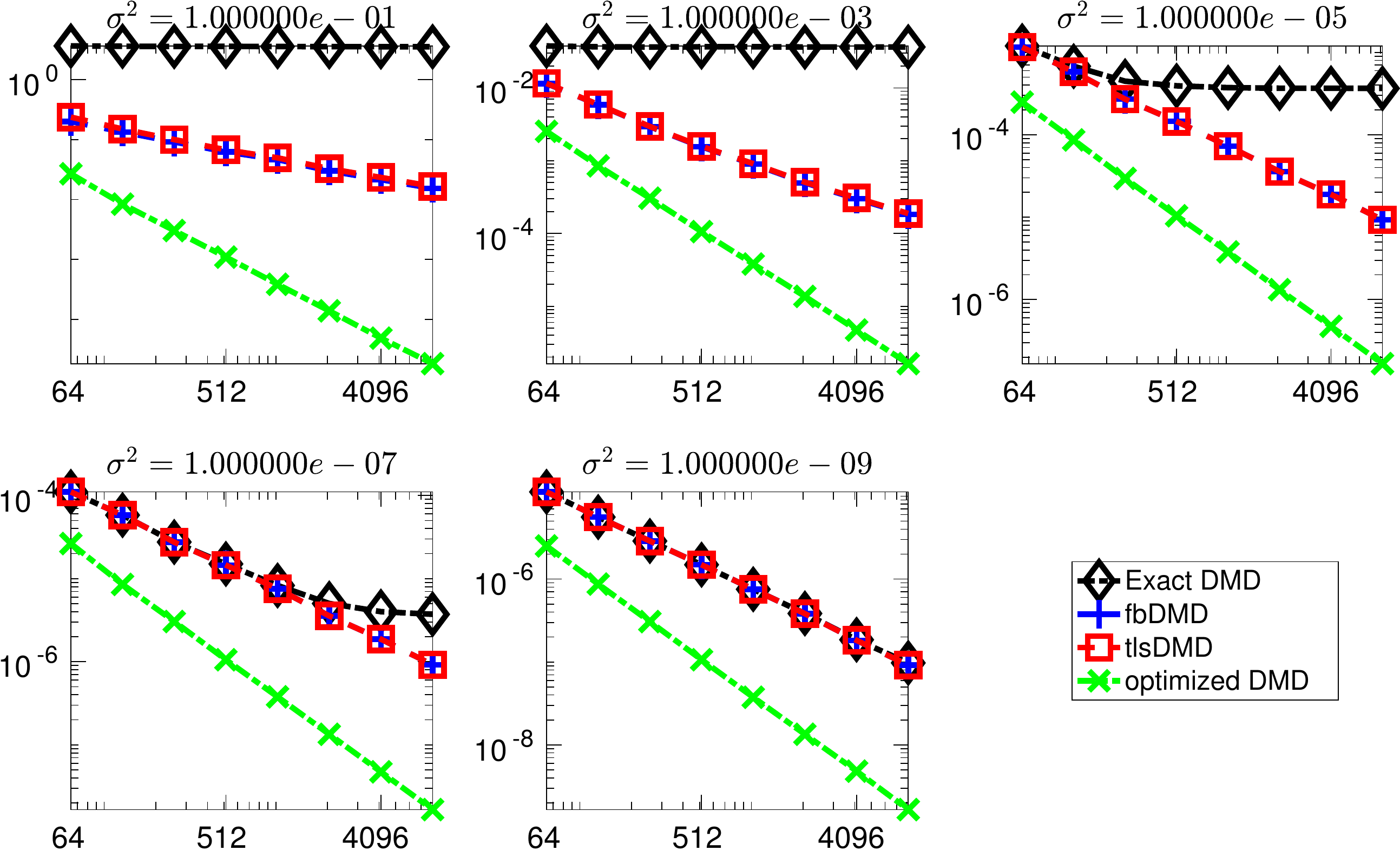}

\caption{Example 1. This figure shows the mean $l^2$ error (averaged
over 1000 runs) in the recovered eigenvalues of the
system matrix ${\mathbf A}$ as a function of the number of snapshots
$m$ for various noise levels $\sigma^2$.}

\label{fig:senper_eig}

\end{figure}

\Cref{fig:senper_eig} contains plots of the $l^2$
norm error in the computed eigenvalues (averaged
over the trials) as a function of the
number of snapshots for various noise levels.
Again, the error in the
exact DMD eventually levels off at the lower noise
levels because of the bias in its eigenvalues. The
other methods perform well, with the error decaying
as the number of snapshots increases. However,
in this measure, the advantage of the optimized DMD
is more pronounced. The error in the eigenvalues
for the optimized DMD is lower than for the fbDMD
and tlsDMD across all noise levels and is observed to
decrease faster as the number of snapshots is
increased. 

\begin{figure}[h!]

\centering

\includegraphics[width=.8\textwidth]{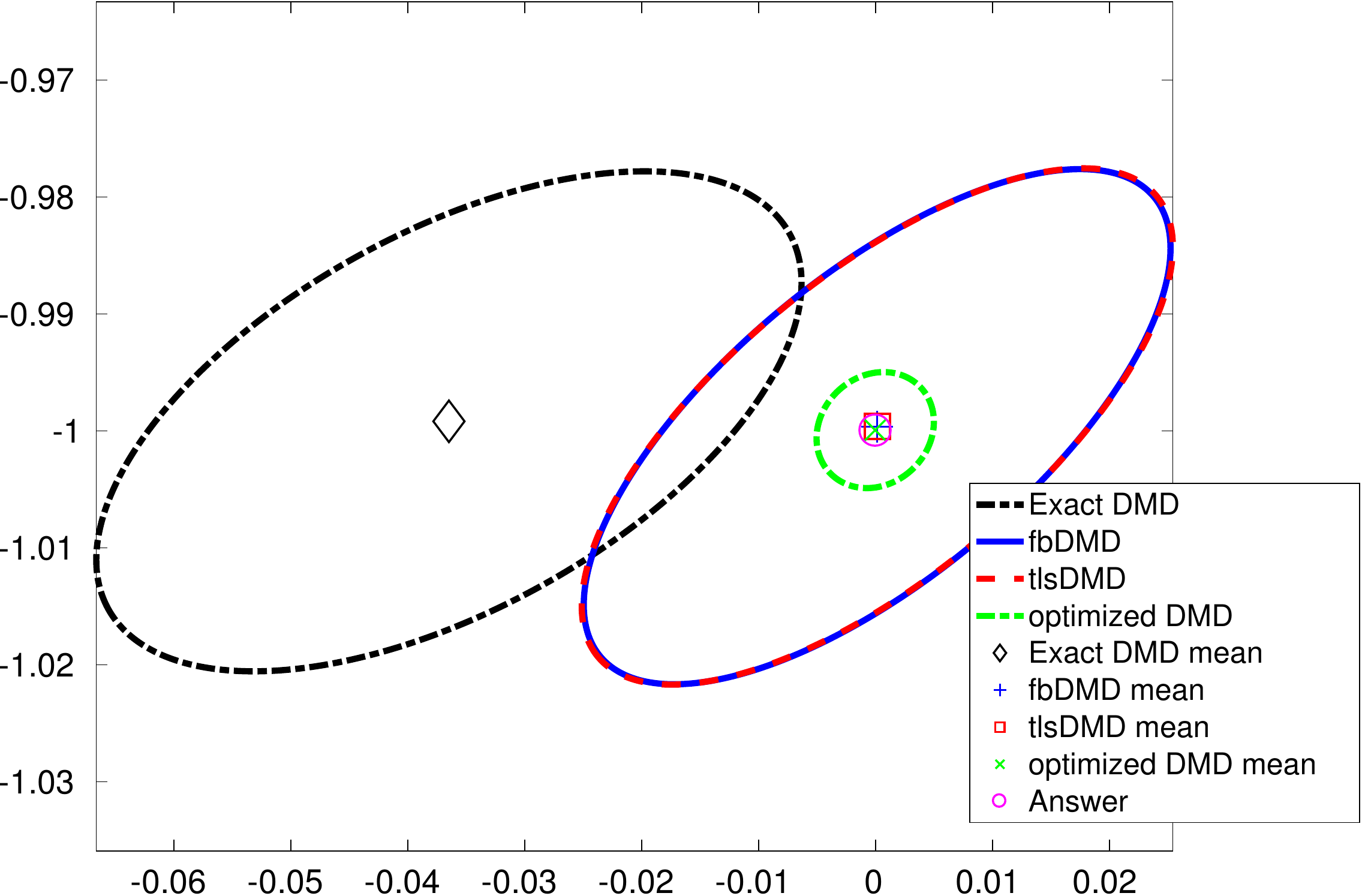}

\caption{Example 1. This figure shows 95 percent confidence ellipses
for the eigenvalue $i$ (based on 1000 runs) for the second-highest
noise level $\sigma^2 = 10^{-3}$ and fewest snapshots
$m = 64$.}

\label{fig:senper_eig_conf}

\end{figure}

We plot 95 percent confidence ellipses (in the complex plane)
for the eigenvalue $i$
for the second-highest noise level and fewest number of snapshots
in \cref{fig:senper_eig_conf}. The bias in the exact DMD is
evident, as the center of the ellipse is seen to be shifted into
the left half-plane. The fbDMD and tlsDMD are seen to correct
for this bias, but the spread of the optimized DMD eigenvalues
is notably smaller.

\begin{figure}[h!]

\centering

\includegraphics[width=.8\textwidth]{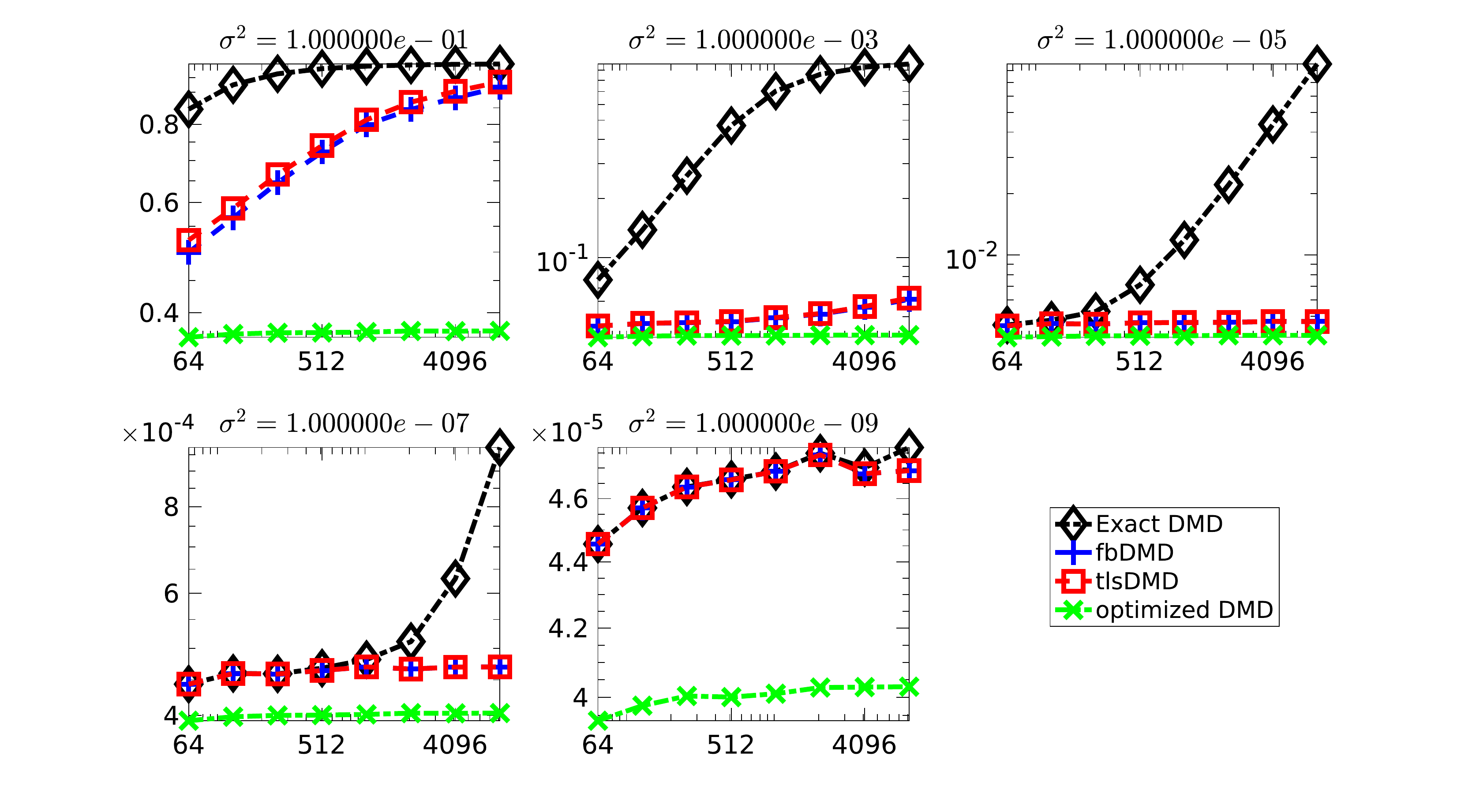}

\caption{Example 1. This figure shows the mean Frobenius norm error 
(averaged over 1000 runs) of the reconstructed snapshots
as a function of the number of snapshots $m$ for various 
noise levels $\sigma^2$.}

\label{fig:senper_rec}

\end{figure}

In \cref{fig:senper_rec}, we plot the mean error in the optimal
reconstruction of the snapshots using the computed eigenvalues,
see \cref{eq:optrecerr} for the definition of this error.
The optimized DMD demonstrates an advantage
across noise levels and number of snapshots used, though
it is more pronounced at higher noise levels. Further,
the error in the optimized DMD is seen to be relatively
flat as the number of snapshots increases, which is to
be expected. The reconstruction error increases for the
other methods as the number of snapshots increases, particularly
at the higher noise levels.

\subsubsection{Example 2: measurement noise, hidden dynamics}

In the case that a signal contains some rapidly decaying
components it can be more difficult to identify  the
dynamics, particularly in the presence of sensor noise
\cite{dawson2016}. We consider a signal composed of two
sinusoidal signals which are translating, with one growing
and one decaying, i.e.

\begin{equation}
  z(x,t) = \sin( k_1x - \omega_1 t) e^{\gamma_1 t} +
  \sin( k_2x - \omega_2 t) e^{\gamma_2 t} \; ,
\end{equation}
where $k_1 = 1$, $\omega_1 = 1$, $\gamma_1 = 1$, $k_2 = 0.4$,
$\omega_2 = 3.7$, and $\gamma_2 = -0.2$ (these are the
settings used in \cite{dawson2016}). This signal
has four continuous time eigenvalues given
by $\gamma_1 \pm i \omega_1$ and $\gamma_2 \pm i \omega_2$.
We set the domain of $x$ to be $[0,15]$ and use
300 equispaced points to discretize. For the
time domain, we set $\Delta t = 2\pi/(2^9-1)$
so that the largest number of snapshots we use,
$m = 2^9$, covers $[0,2\pi]$.

As before, we consider the effect of both the size
of the noise, $\sigma$, and the number of snapshots,
$m$, on the
quality of the modes and eigenvalues obtained
from various methods. We set the noise level to the
values $\sigma^2 = 2^{-2}, 2^{-4}, \ldots, 2^{-10}$
and run tests with $m = 2^7, 2^8, 2^{9}$
snapshots (the range of the number of snapshots
is more limited for this problem by the exponential
growth in the signal).
For each noise level and number of snapshots,
we compute the eigenvalues and modes of this system
using the exact DMD, fbDMD, tlsDMD, and optimized
DMD over 1000 trials (different draws of the vector
${\mathbf g}$). We compute the DMD for each of these methods
with the data
projected on the first 4 POD modes (the first 4
left singular vectors of the data matrix). For the
optimized DMD, this means we are using the
approximate algorithm, \cref{algo:approxoptdmd}.

\begin{figure}[h!]

\centering

\includegraphics[width=.8\textwidth]{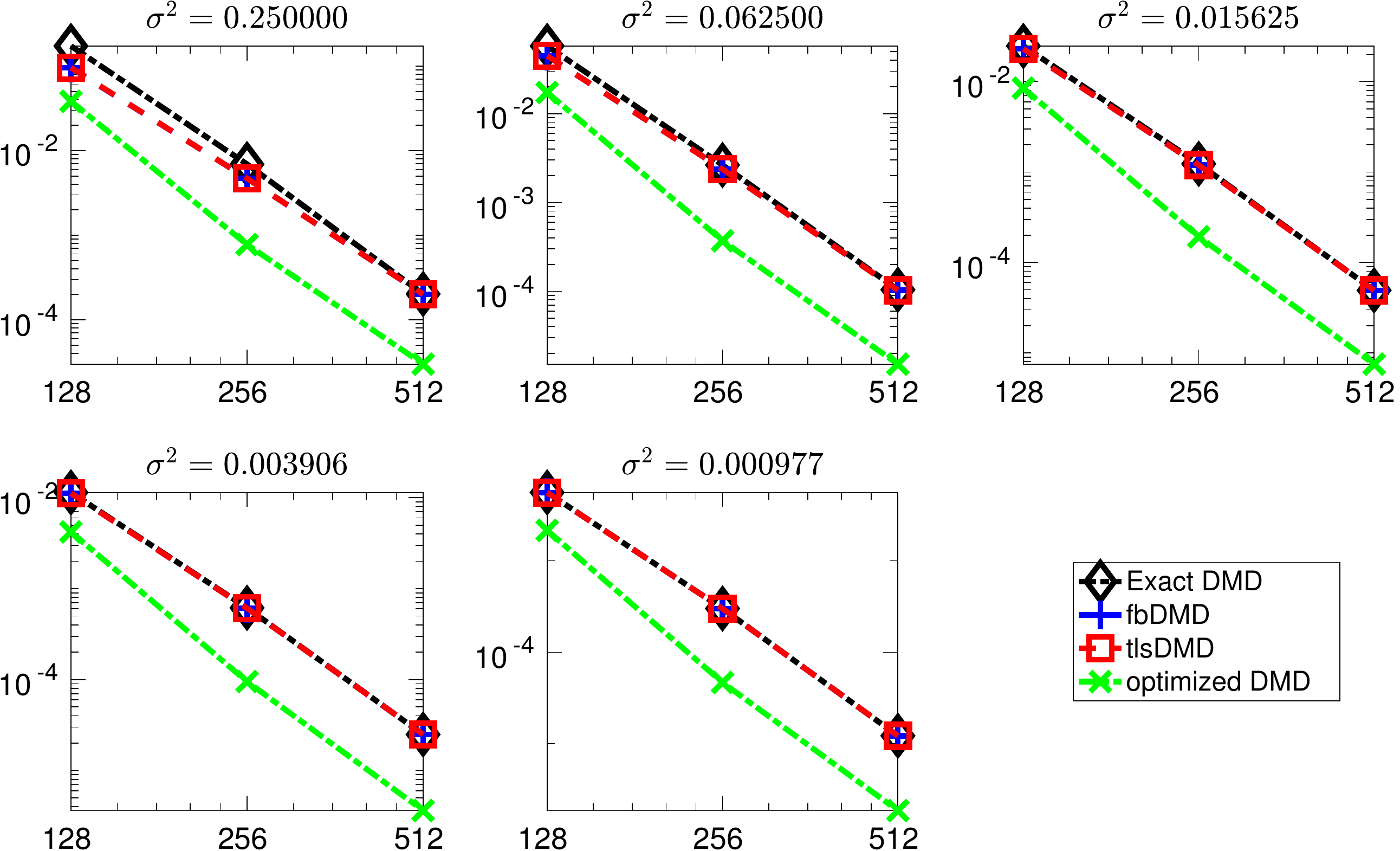}

\caption{Example 2. This figure shows the mean $l^2$ error (averaged
over 1000 runs) in the recovered dominant (growing) eigenvalues 
of the system as a function of the number of snapshots
$m$ for various noise levels $\sigma^2$.}

\label{fig:senhid_dom_eig}

\end{figure}

\begin{figure}[h!]

\centering

\includegraphics[width=.8\textwidth]{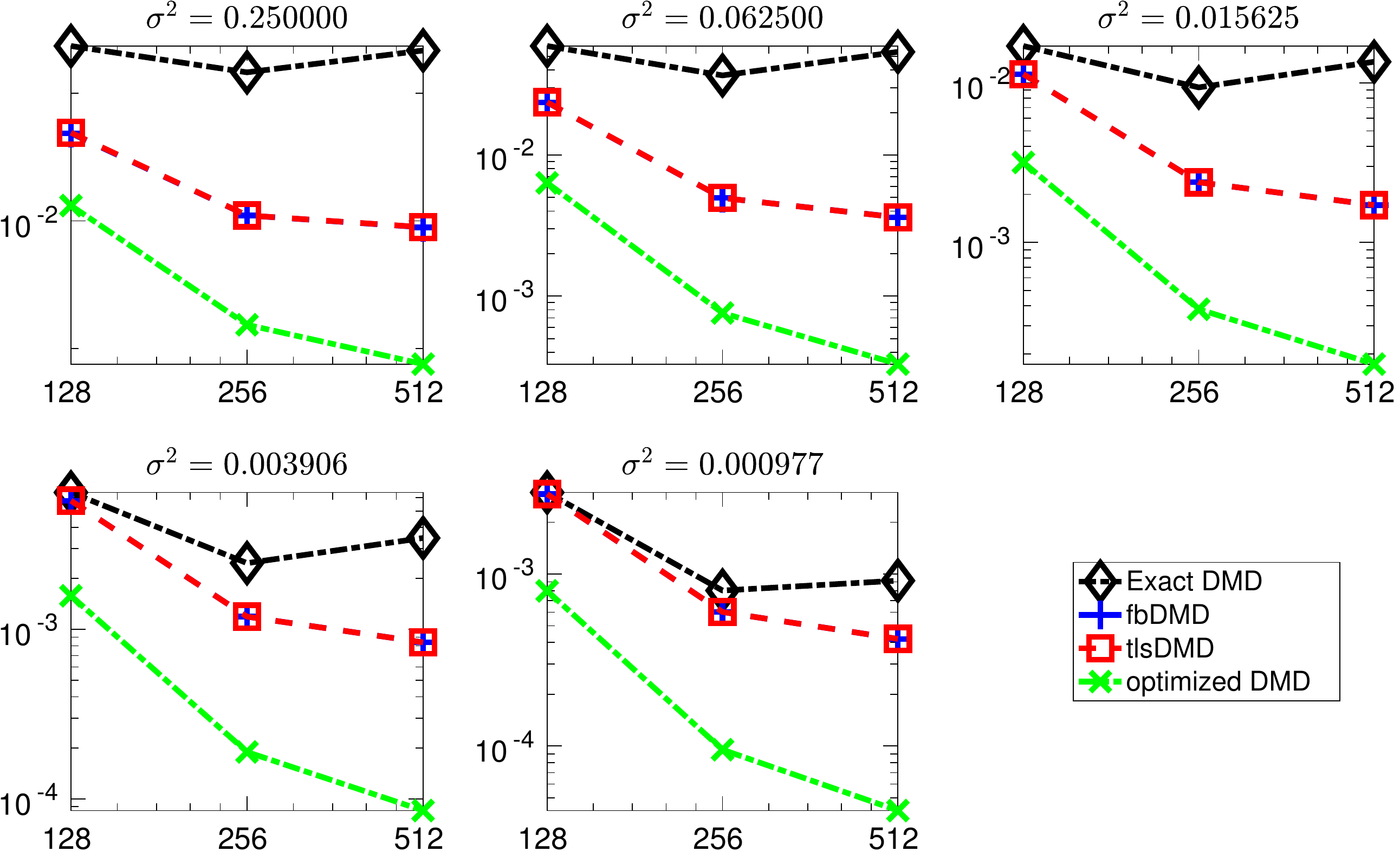}

\caption{Example 2. This figure shows the mean $l^2$ error (averaged
over 1000 runs) in the recovered hidden (shrinking) eigenvalues 
of the system as a function of the number of snapshots
$m$ for various noise levels $\sigma^2$.}

\label{fig:senhid_hid_eig}

\end{figure}

In \cref{fig:senhid_dom_eig,fig:senhid_hid_eig}, we plot
the mean $l^2$ error (averaged over the trials) in the
recovered dominant eigenvalues ($1\pm i$) and the
recovered hidden eigenvalues ($-0.2 \pm 3.7 i$),
respectively. For the dominant eigenvalues, the exact
DMD, fbDMD, and tlsDMD perform similarly and the optimized
DMD has lower error, up to an order of magnitude
more accurate for some settings.
For the hidden eigenvalues, the bias in the exact DMD
is evident and it performs notably worse. It seems
that the same phenomenon occurs in which the bias
of the exact DMD causes these curves to level off
prematurely. 
Of course, the signal for the hidden eigenvalues will
eventually decay so much that there is no advantage
in adding more snapshots. This is evident in these
plots as the curves flatten out for all methods. 
Again, the optimized DMD has an advantage across
noise levels and number of snapshots. Importantly,
it appears that for some noise levels, the other
methods won't perform as well as the optimized DMD,
even if you provide them with more snapshots.

\begin{figure}[h!]

\centering

\includegraphics[width=.8\textwidth]{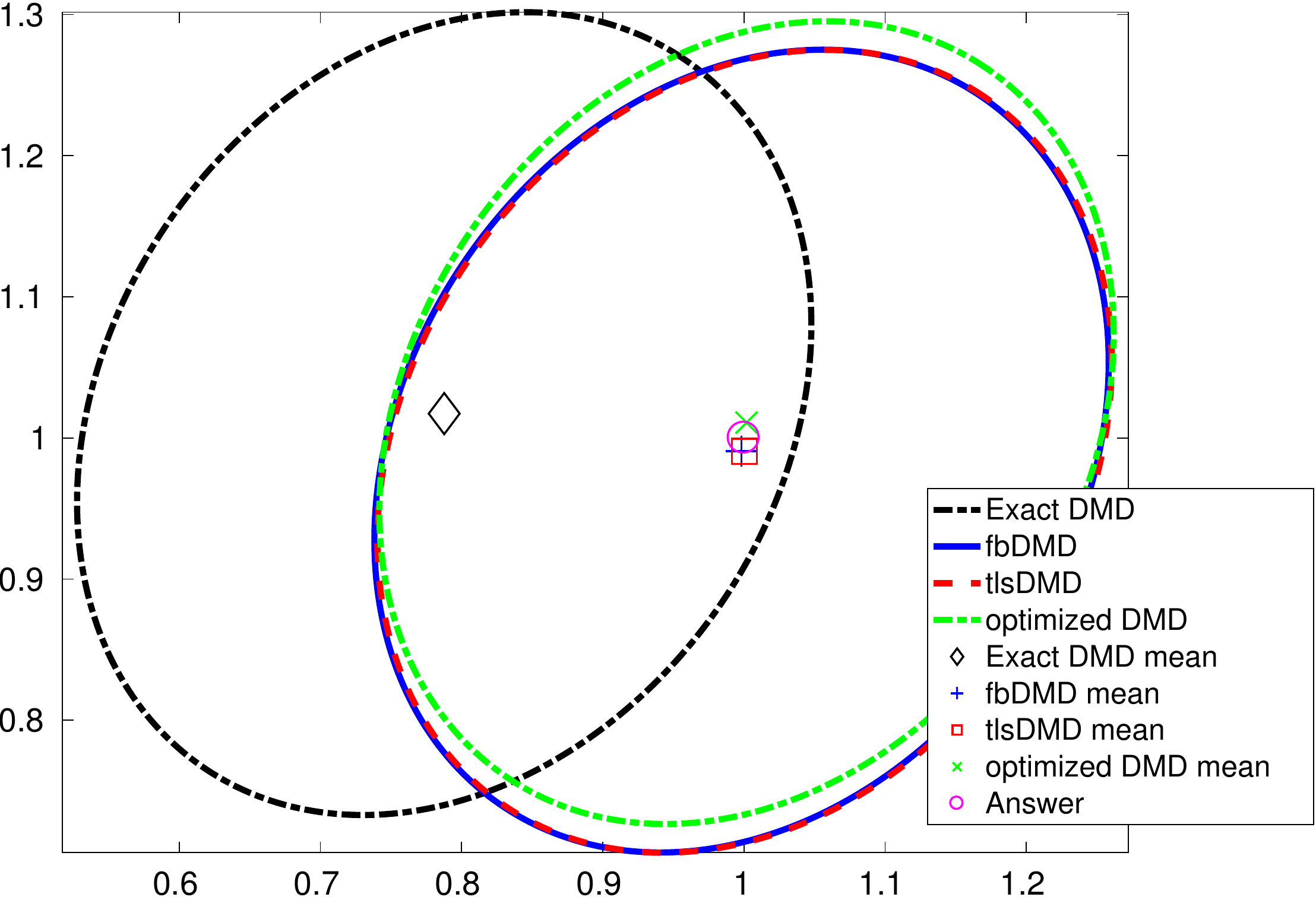}

\caption{Example 2. This figure shows 95 percent confidence ellipses
for one of the dominant eigenvalues (based on 1000 runs) 
for the highest noise level $\sigma^2 = 2^{-2}$ 
and fewest snapshots $m = 128$.}

\label{fig:senhid_dom_conf}

\end{figure}

\begin{figure}[h!]

\centering

\includegraphics[width=.8\textwidth]{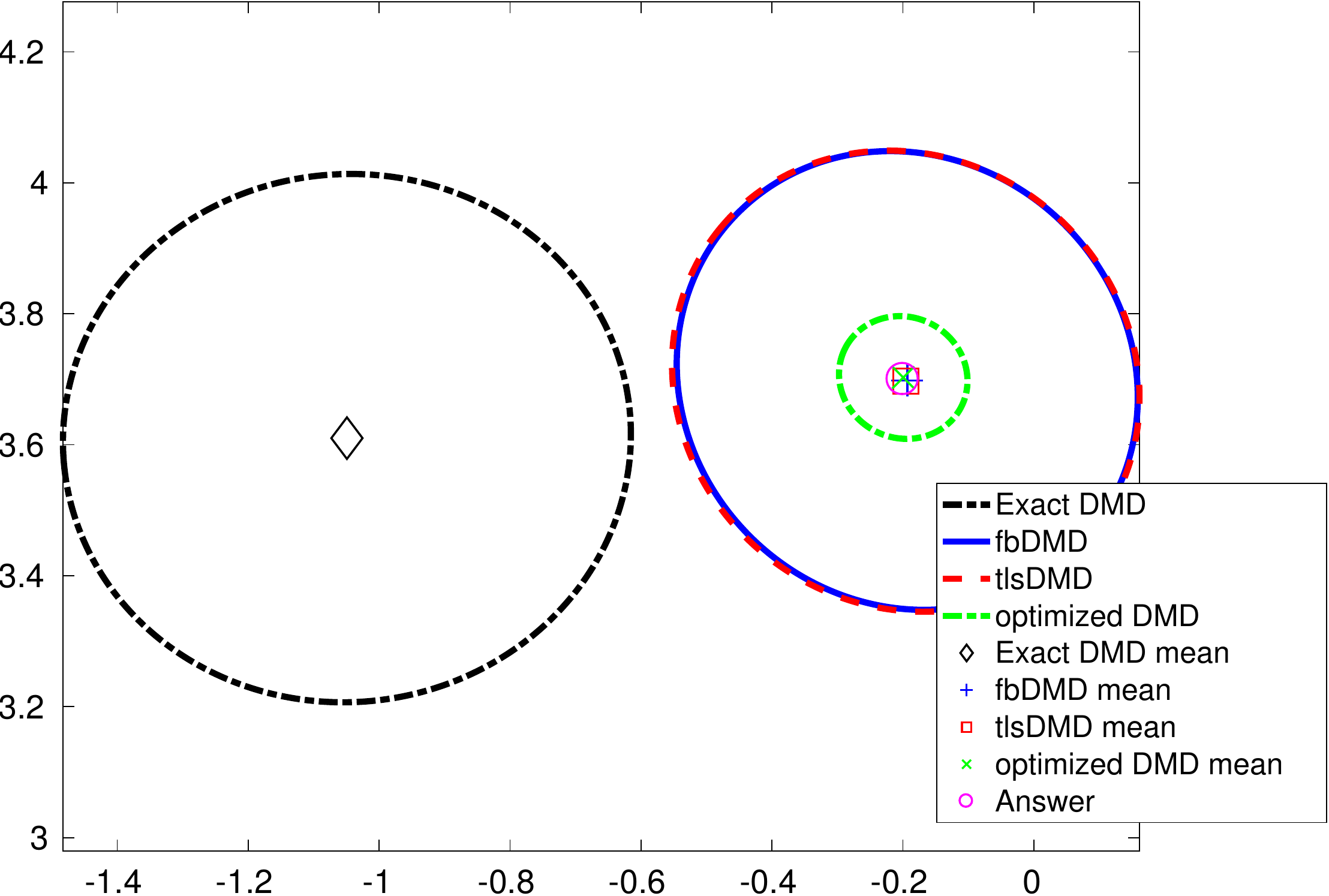}

\caption{Example 2. This figure shows 95 percent confidence ellipses
for one of the hidden eigenvalues (based on 1000 runs) 
for the highest noise level $\sigma^2 = 2^{-2}$ 
and fewest snapshots $m = 128$.}

\label{fig:senhid_hid_conf}

\end{figure}

We plot 95 percent confidence ellipses (in the complex plane)
for the eigenvalues $1+i$ and $-0.2+3.7i$
in \cref{fig:senhid_dom_conf,fig:senhid_hid_conf}, respectively,
for the highest noise level and fewest number of snapshots.
The bias in the exact DMD is
evident, as the center of the ellipse is seen to be shifted
to the left for both eigenvalues. For the dominant eigenvalue
($1+i$),
we see that the spread of the eigenvalues for the fbDMD,
tlsDMD, and optimized DMD are similar and that all three
of these methods correct for the bias. For the hidden
eigenvalue ($-0.2+3.7i$), the fbDMD, tlsDMD, and optimized
DMD all correct for the bias but the spread of the
optimized DMD is notably smaller.

\begin{figure}[h!]

\centering

\includegraphics[width=.8\textwidth]{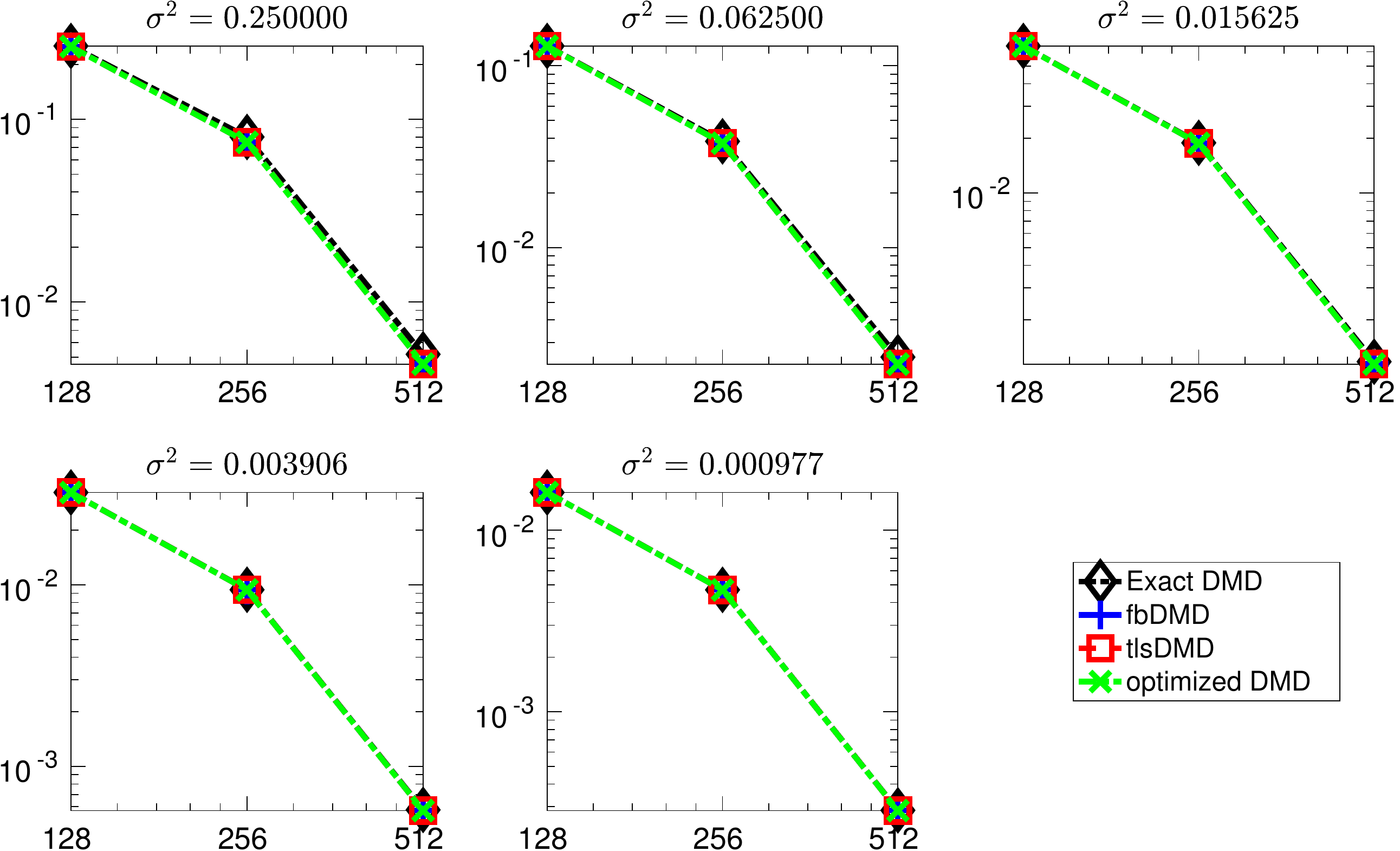}

\caption{Example 2. This figure shows the mean Frobenius norm error 
(averaged over 1000 runs) of the reconstructed snapshots
as a function of the number of snapshots $m$ for various 
noise levels $\sigma^2$.}

\label{fig:senhid_rec}

\end{figure}

In \cref{fig:senhid_rec}, we plot the mean error in the optimal
reconstruction of the snapshots using the computed eigenvalues,
see \cref{eq:optrecerr} for the definition of this error. In
contrast with the periodic example, the error curves roughly
coincide for all methods and the error decreases as the number
of snapshots increases. This is likely a result of the fact
that the growing modes dominate more and more as the
system advances in time (when new snapshots are added, they
come from later in the time series). It is interesting that
the reconstruction error is only marginally better for the
optimized DMD --- this error is what the optimized DMD tries to
minimize --- but the recovered eigenvalues are significantly
better.

\begin{figure}[h!]

\centering

\includegraphics[width=.8\textwidth]{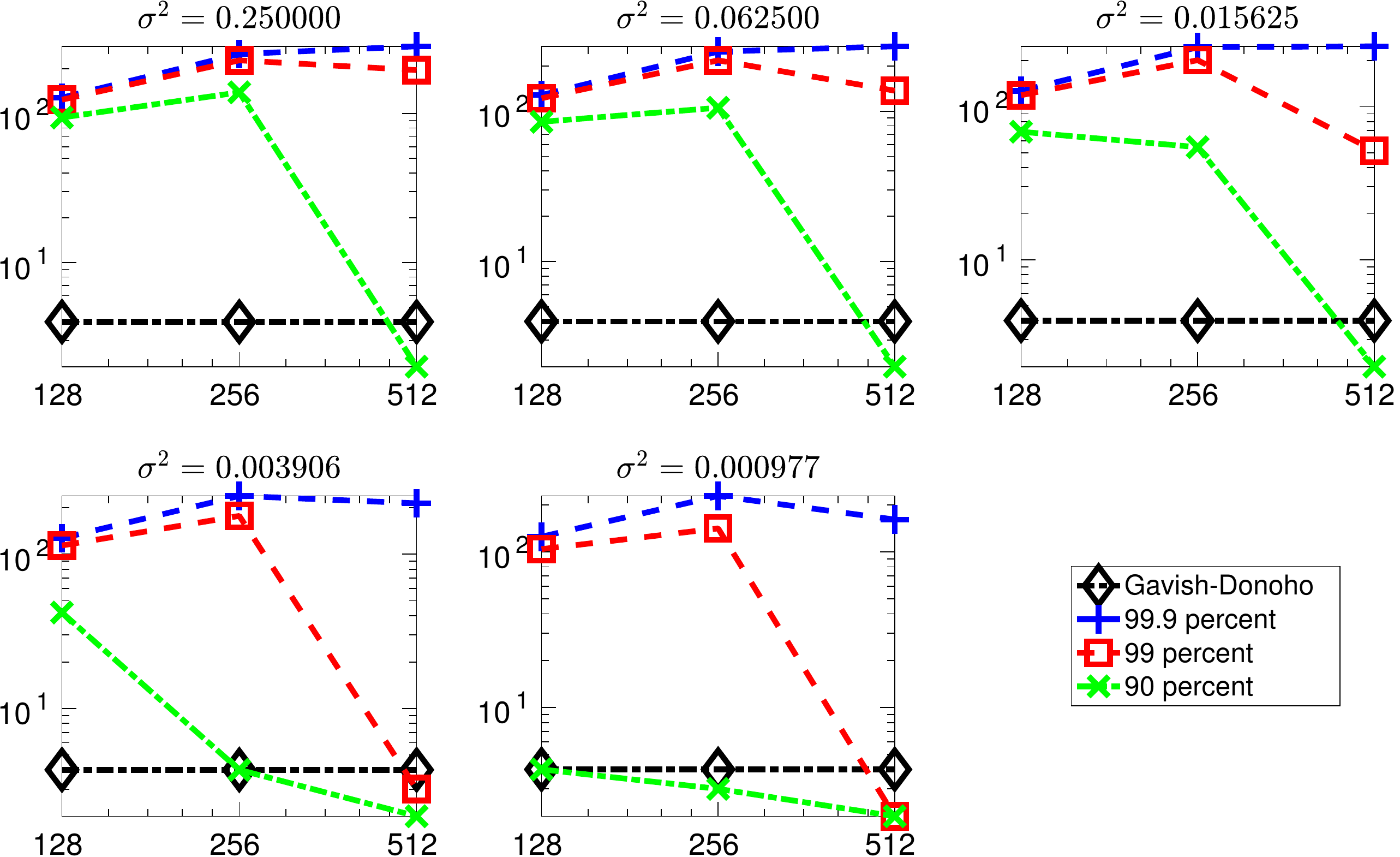}

\caption{Example 2. This figure shows the mean estimated
  rank of the data (averaged over 1000 runs) using
  the Gavish-Donoho, 99.9 percent, 99 percent, and 90
  percent hard-thresholds
  as a function of the number of snapshots $m$ for various 
  noise levels $\sigma^2$.}

\label{fig:senhid_rank}

\end{figure}

Because the data in this example is in a high dimensional
space relative to the rank of the dynamics, we must use
some sort of truncation when computing the DMD (using
any of the methods). For the comparisons above, we used
the a priori knowledge we have of the system
to always truncate at rank 4. When this a priori
information is unavailable, it is sometimes necessary
to determine an appropriate truncation from the data.
In \cref{fig:senhid_rank}, we compare the hard-threshold
(the number of singular values to keep) obtained from
the Gavish-Donoho formula \cite{gavish2014} with the
hard-threshold obtained by
keeping 99.9, 99, and 90 percent of the energy in the
singular values. The Gavish-Donoho formula always
produced 4 in our experiments, while the cut-offs based
on keeping a certain percentage of the energy produced
wildly different results depending on noise level and
number of snapshots. The type of error in this
example exactly satisfies the assumptions used to
obtain the Gavish-Donoho formula; nonetheless, the
performance of the formula is impressive in comparison
with these other a posteriori methods.

\begin{figure}[h!]

\centering

\includegraphics[width=.8\textwidth]{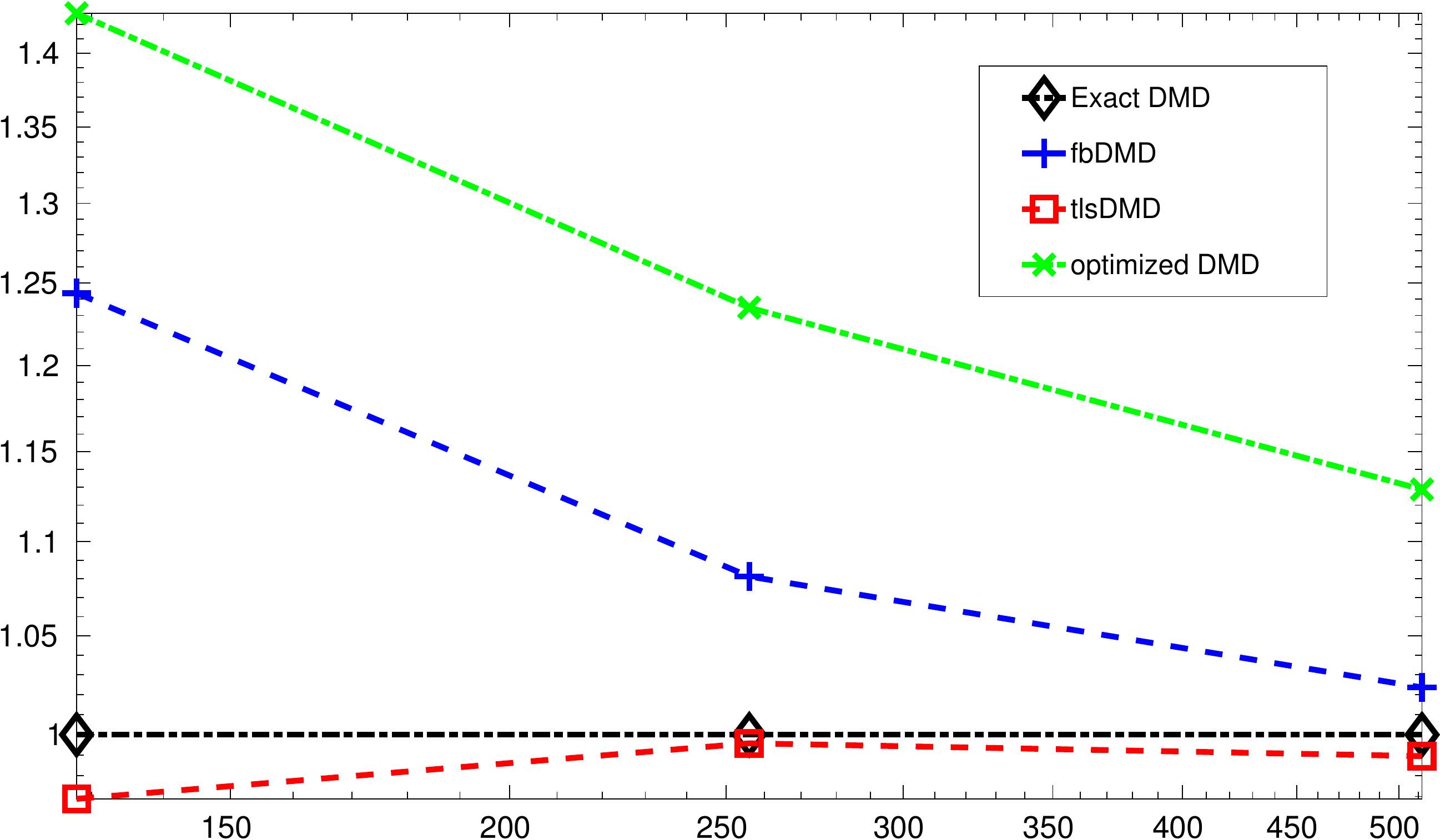}

\caption{Example 2. This figure shows the mean run-time
  of the methods (averaged over 1000 runs) relative to
  the mean run-time for the exact DMD 
  as a function of the number of snapshots $m$.}

\label{fig:senhid_times}

\end{figure}

In \cref{fig:senhid_times}, we plot the mean
run-time of each method as you increase the number
of snapshots, with the values normalized by the
mean run-time of the exact DMD. We see that the
optimized DMD indeed requires more computation than
the other methods but that the increase is modest.
For this example, the run-time is dominated by the
SVD used for the truncation in each method.
These values should be taken with a grain of salt,
as they depend significantly on the quality of
the implementation (and, for the optimized DMD, on
the parameters sent to the optimization routine). 
Note that our implementation
of the fbDMD checks every possible square root for
the optimal answer, which is costly for larger
systems.

\subsubsection{Example 3: uncertain sample times, periodic system}

For this example, we revisit the system of
Example 1, \cref{eq:example1}, but introduce
a different type of sampling error: uncertain
sample times. 
Let ${\mathbf z}(t)$ be the solution of \cref{eq:example1}
with the initial condition ${\mathbf z}(0) = (1,0.1)^\intercal$.
Let the snapshots be given by ${\mathbf z}_j = {\mathbf z}( (j+\sigma g_j) \Delta t)$
with $\Delta t = 0.1$, $\sigma$ a prescribed noise level, and
${\mathbf g}$ a vector whose entries are drawn from a standard
normal distribution. 
Again, the continuous time eigenvalues of this system
are $\pm i$ (this is how the optimized DMD computes
eigenvalues) and the discrete time eigenvalues are
$\exp(\pm \Delta t i)$. Intuitively, the methods should
behave like they did for the sensor noise
example (consider the Taylor series of
${\mathbf z}( (j+\sigma g_j) \Delta t)$ about ${\mathbf z}(j\Delta t)$) but
there is a different structure to the noise here.

We consider the effect of both the size of the noise,
$\sigma$, and the number of snapshots, $m$, on the
quality of the modes and eigenvalues obtained
from various methods. We set the noise level to the
values $\sigma^2 = 2^{-2}, 2^{-4}, \ldots, 2^{-10}$
and run tests with $m = 2^6, 2^7, \ldots, 2^{13}$
snapshots. 
For each noise level and number of snapshots,
we compute the eigenvalues and modes of this system
using the exact DMD, fbDMD, tlsDMD, and optimized
DMD over 1000 trials (different draws of the vector
${\mathbf g}$). 

\begin{figure}[h!]

\centering

\includegraphics[width=.8\textwidth]{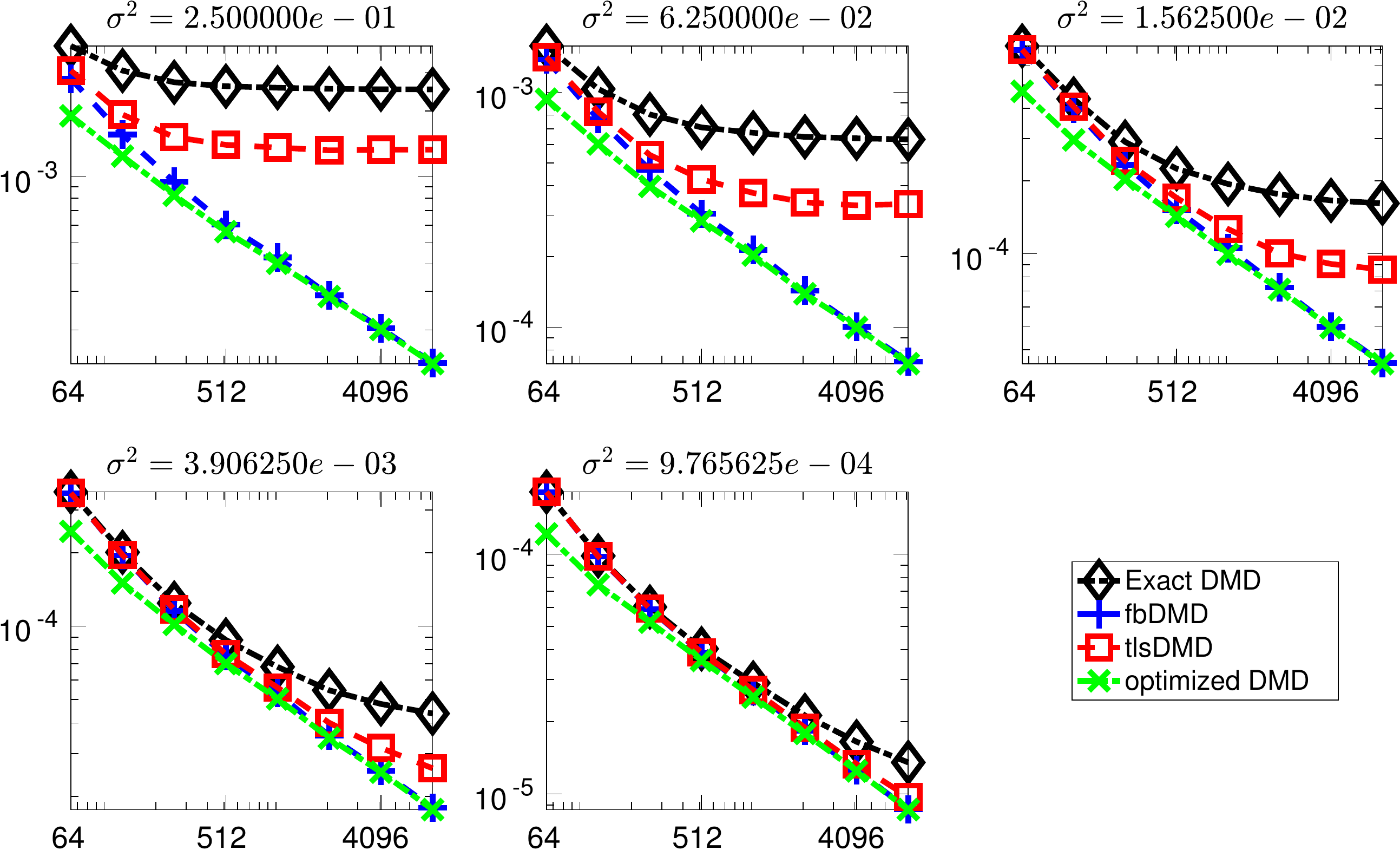}

\caption{Example 3. This figure shows the mean Frobenius norm error 
(averaged over 1000 runs) in the reconstructed
system matrix ${\mathbf A}$ as a function of the number of snapshots
$m$ for various noise levels $\sigma^2$.}

\label{fig:samper_reca}

\end{figure}

In \cref{fig:samper_reca}, we plot the mean Frobenius
norm error in the reconstructed system matrix (averaged
over the trials) as a function of the
number of snapshots for various noise levels.
We see that, as in Example 1, the error in the
exact DMD eventually levels off at the higher noise
levels because of the bias in its eigenvalues. Surprisingly,
this occurs for the tlsDMD as well, but at a lower
error. The fbDMD and optimized DMD perform well, with
the error decaying
as the number of snapshots increases. The optimized
DMD shows slight improvement over the fbDMD 
at the highest noise levels and with the
fewest snapshots. 

\begin{figure}[h!]

\centering

\includegraphics[width=.8\textwidth]{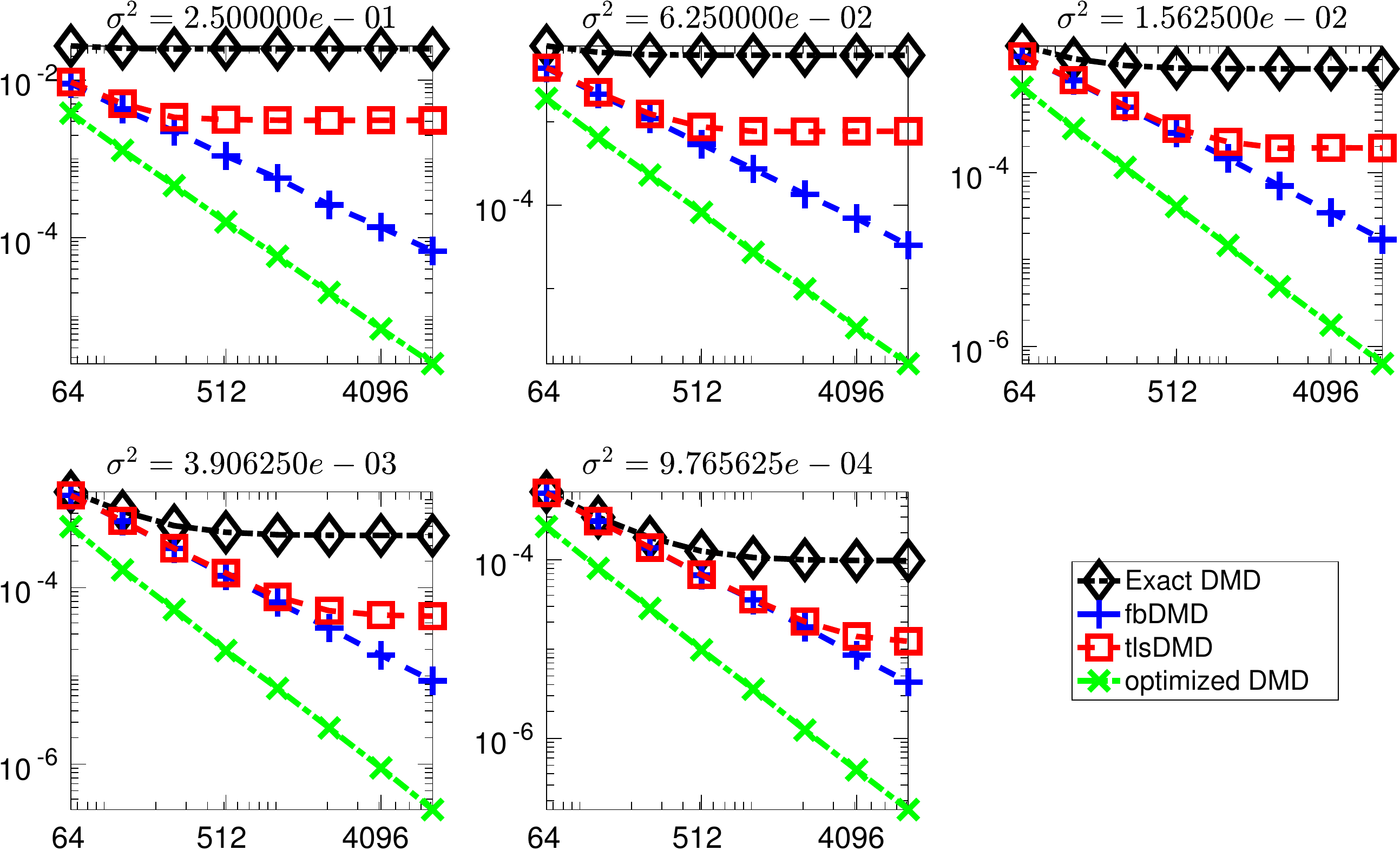}

\caption{Example 3. This figure shows the mean $l^2$ error (averaged
over 1000 runs) in the recovered eigenvalues of the
system matrix ${\mathbf A}$ as a function of the number of snapshots
$m$ for various noise levels $\sigma^2$.}

\label{fig:samper_eig}

\end{figure}

\Cref{fig:samper_eig} contains plots of the $l^2$
norm error in the computed eigenvalues (averaged
over the trials) as a function of the
number of snapshots for various noise levels.
Again, the error in the
exact DMD eventually levels off at the lower noise
levels because of the bias in its eigenvalues. We see
similar behavior for the tlsDMD. The
fbDMD and optimized DMD perform well, with the error decaying
as the number of snapshots increases. However,
in this measure, the advantage of the optimized DMD
is more pronounced. The error in the eigenvalues
for the optimized DMD is lower than for the fbDMD
and tlsDMD across all noise levels and is observed to
decrease faster as the number of snapshots is
increased.

\begin{figure}[h!]

\centering

\includegraphics[width=.8\textwidth]{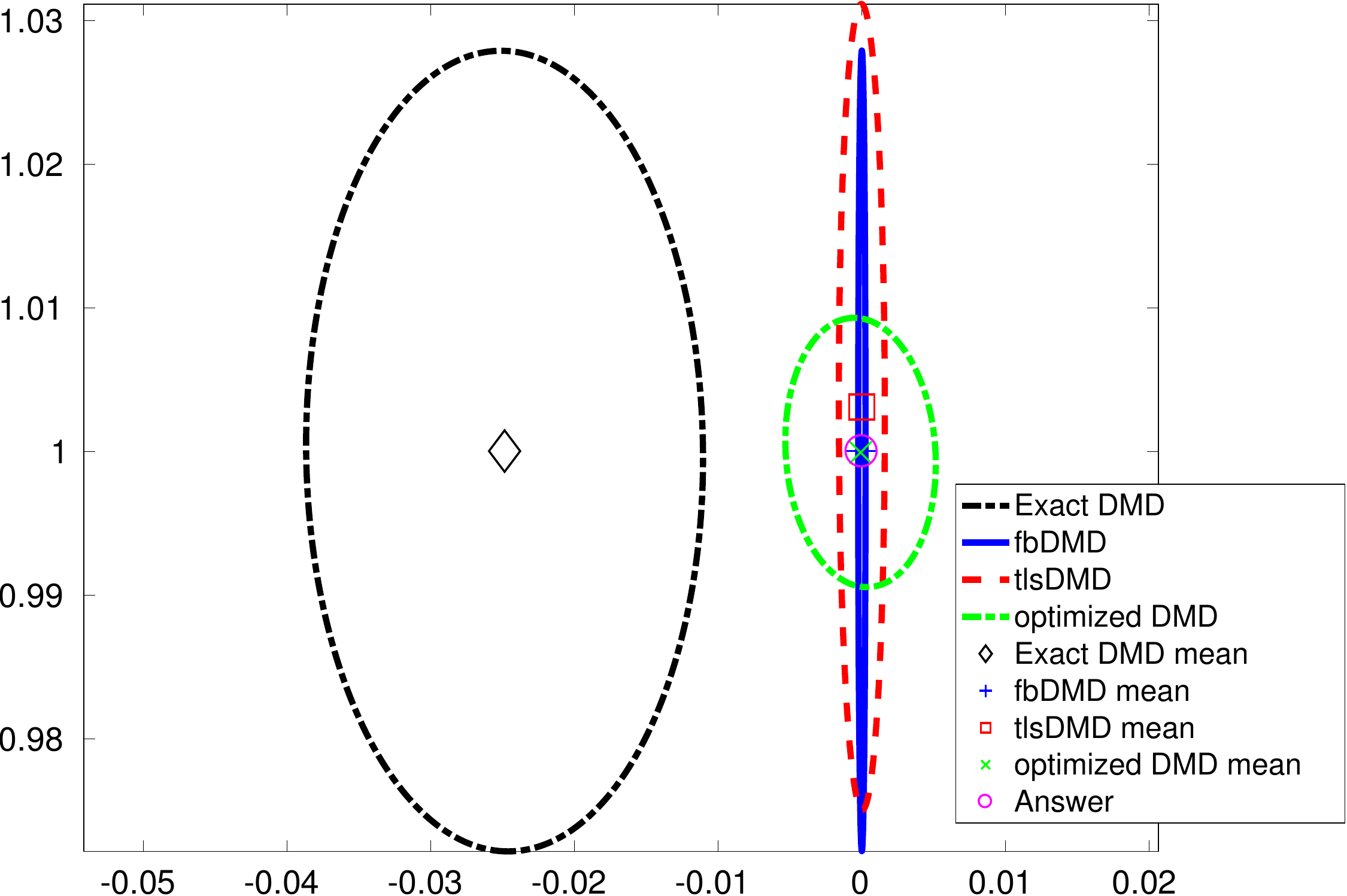}

\caption{Example 3. This figure shows 95 percent confidence ellipses
for the eigenvalue $i$ (based on 1000 runs) for the highest
noise level $\sigma^2 = 2^{-2}$ and fewest snapshots
$m = 64$.}

\label{fig:samper_eig_conf}

\end{figure}

We plot 95 percent confidence ellipses (in the complex plane)
for the eigenvalue $i$
for the highest noise level and fewest number of snapshots
in \cref{fig:samper_eig_conf}. Again, the bias in the exact DMD is
indicated by the fact that the ellipse is shifted into
the left half-plane. Curiously, the tlsDMD displays a different
type of bias, consistently overestimating the frequency of
the oscillation.
The fbDMD and optimized DMD are relatively bias-free, with
the fbDMD having a smaller spread along the real axis and the
optimized DMD having a smaller spread along the imaginary axis.
We believe that the strong performance of the fbDMD here is
rather intuitive: by averaging the forward and backward dynamics,
the fbDMD should nearly cancel the noise we've introduced with the
uncertain sample times. 

\begin{figure}[h!]

\centering

\includegraphics[width=.8\textwidth]{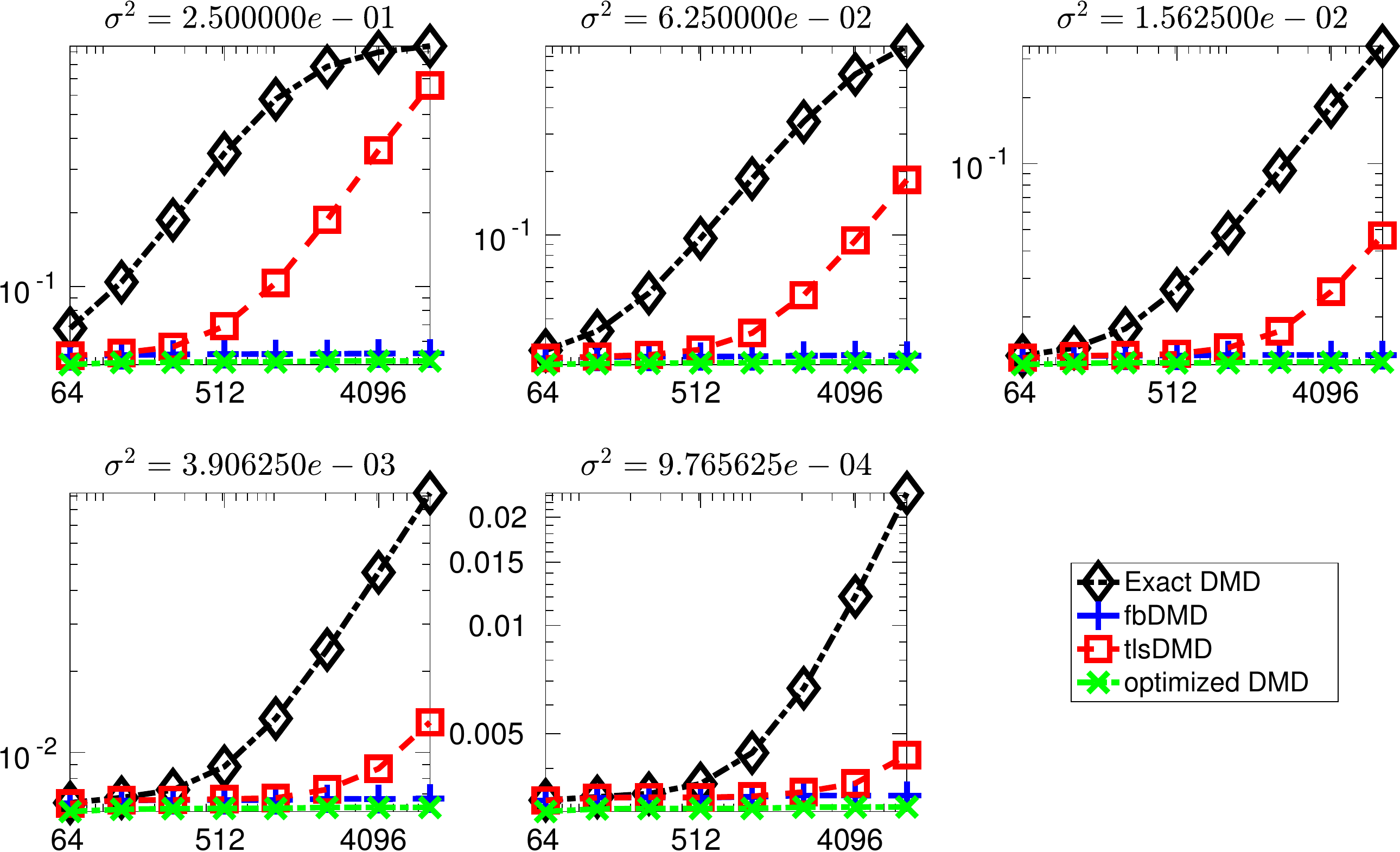}

\caption{Example 3. This figure shows the mean Frobenius norm error 
(averaged over 1000 runs) of the reconstructed snapshots
as a function of the number of snapshots $m$ for various 
noise levels $\sigma^2$.}

\label{fig:samper_rec}

\end{figure}

In \cref{fig:samper_rec}, we plot the mean error in the optimal
reconstruction of the snapshots using the computed eigenvalues,
see \cref{eq:optrecerr} for the definition of this error.
The fbDMD and optimized DMD perform the best
across noise levels and number of snapshots used, though
the advantage is more pronounced at higher noise levels.
The reconstruction error increases for the exact DMD and tlsDMD
as the number of snapshots increases, particularly
at the higher noise levels.

From the above, we see that uncertain sample times can
produce errors in the computed DMD modes and eigenvalues
which are qualitatively different from the errors produced
by sensor noise. We believe that this source of error
may be of interest when analyzing data collected by
humans or historical data sets. When dealing with real
data, it is clear how to perform sensitivity analysis
for additive sensor noise: simply rerun the method for
the data with sensor noise added. It is unclear how to
perform sensitivity analysis for uncertain sample times
using the exact DMD, fbDMD, or tlsDMD. Using the
optimized DMD, such an analysis is again simple: rerun
the method for the same data set while adding noise to the
sample times that you send to the optimized DMD routine.

\subsection{Example 4: Sea surface temperature data}

For the final example, we consider a real data set:
the ``optimally-interpolated'' sea surface temperature
(OISST-v2, AVHRR only) data set from the
National Oceanic and Atmospheric Administration (NOAA)
\cite{reynolds2007,banzon2016}.
This is a data set of daily average ocean temperatures,
with 1/4 degree resolution in latitude and longitude,
for a total of about 700k grid points over the ocean.
The temperatures are reported to four decimal digits.
We considered two subsets of this data: 521 snapshots
(10 years) spaced 7 days apart and 521 snapshots spaced
randomly, with an average of 7 days apart, with each set
starting on January 1st, 1982.

\begin{table}
  \centering
  \label{tab:sst_rec}
  \begin{tabular}{|c|c|c|c|c|}                                             
\hline                                                                   
rank & $\rho_{DMD}$ & $\rho_{tlsDMD}$ & $\rho_{optDMD}$ & $\rho_{POD}$ \\
\hline                                                                   
2 & 1.0446e-01 & 1.0447e-01 & 1.0437e-01 & 4.8298e-02 \\                 
\hline                                                                   
4 & 6.5921e-02 & 5.3714e-02 & 4.5387e-02 & 4.1057e-02 \\                 
\hline                                                                   
8 & 5.3453e-02 & 4.4864e-02 & 3.8655e-02 & 3.7276e-02 \\                 
\hline                                                                   
16 & 4.6376e-02 & 4.0935e-02 & 3.5443e-02 & 3.3905e-02 \\                
\hline                                                                   
32 & 3.8045e-02 & 3.2955e-02 & 3.1733e-02 & 3.0039e-02 \\                
\hline                                                                   
64 & 3.1939e-02 & 2.8125e-02 & 2.7884e-02 & 2.5879e-02 \\                
\hline                                                                   
128 & 2.7792e-02 & 9.8115e-01 & 2.3751e-02 & 2.1195e-02 \\               
\hline                                                                   
\end{tabular}                                                       

  \caption{Example 4. This table shows the relative
    residual of the best possible reconstruction using
    the eigenvalues obtained from the exact DMD,
    tlsDMD, and optimized DMD, as $\rho_{DMD}$,
    $\rho_{tlsDMD}$, and $\rho_{optDMD}$, respectively,
    for several different values of the reconstruction
    rank $r$. The value $\rho_{POD}$ is the relative norm
    of the residual when 
    projecting the data onto $r$ POD modes and represents a
    rough lower bound on the relative
    residual for DMD modes.
  }

\end{table}
    
\begin{figure}[h!]

\centering

\includegraphics[width=.8\textwidth]{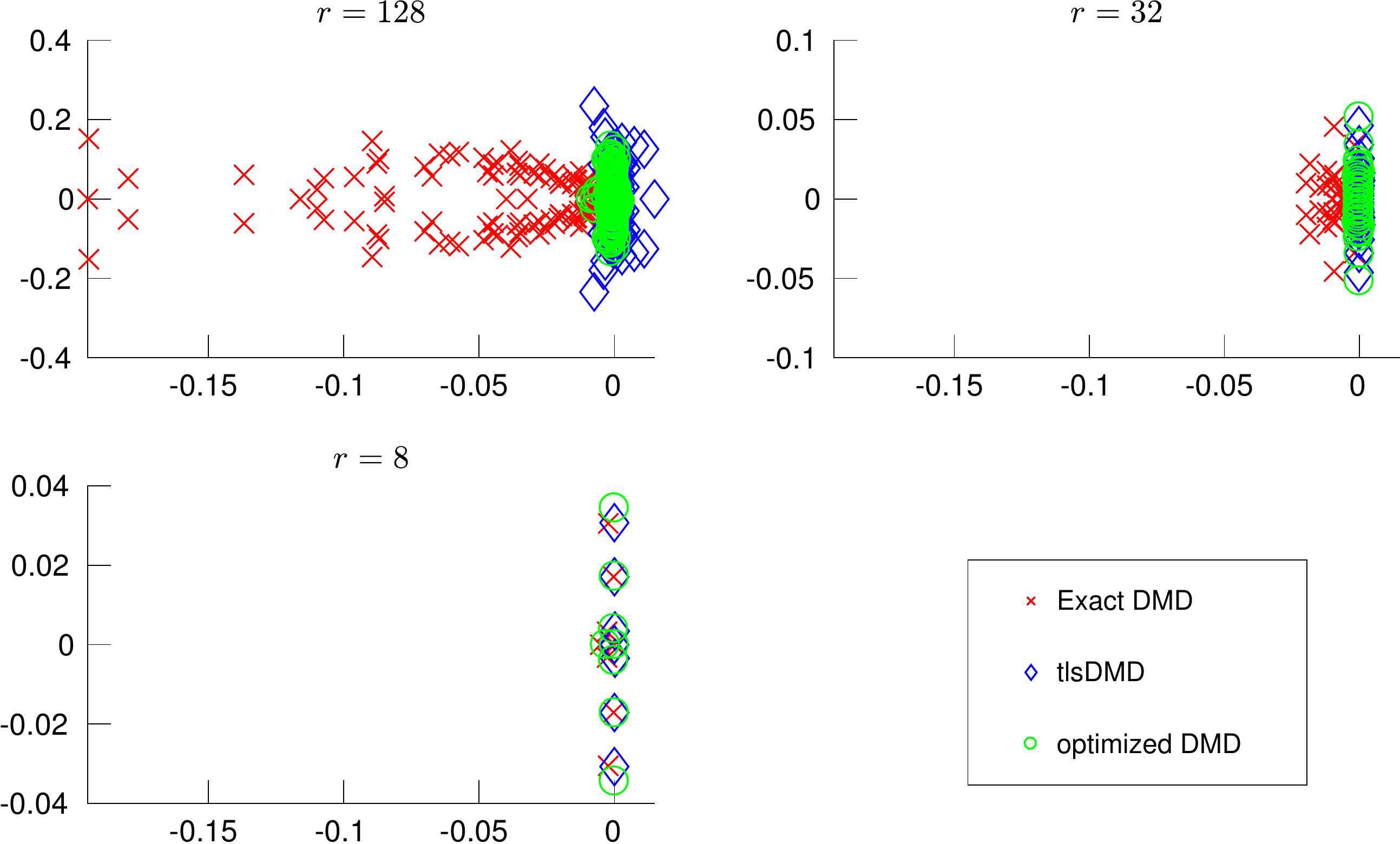}

\caption{Example 4. We plot the eigenvalues 
obtained using the exact DMD, tlsDMD, and optimized 
DMD for various target ranks $r$.}

\label{fig:sst_eig_vs_rank}

\end{figure}

When it comes to properly truncating this data set
for the DMD, there are a number of complicating factors:
the sensor noise is not simply additive white noise,
the values have been interpolated, and the underlying
dynamics are not linear. In particular, the assumptions
used to obtain the Gavish-Donoho formula are not satisfied.
In \cref{tab:sst_rec}, we see that the error in the
optimal reconstruction, using the eigenvalues for
any of the methods, decreases weakly as you increase
the rank of the system. This is largely driven by the
slow decay in the singular values of the data matrix
(compare the reconstruction quality for the
optimized DMD with that obtained
for POD modes). For the largest rank, $r = 128$,
the reconstruction error has actually increased
for the tlsDMD. We will see that this is due to some
spurious eigenvalues in the tlsDMD which correspond
to an unreasonable amount of growth.

In \cref{fig:sst_eig_vs_rank}, we produce scatter plots of
the DMD eigenvalues obtained from the exact DMD,
tlsDMD, and optimized DMD for various choices of the DMD
rank $r$. Setting $r = 128$ (this is close to the
value $r = 124$ obtained from the
Gavish-Donoho formula), the exact DMD
has a number of strongly decaying modes and there are
some growing modes visible for the tlsDMD. There is not much
agreement among the methods at this level. For
$r = 32$, the exact DMD and tlsDMD give more reasonable
values and there is more agreement among the methods
but still some significant discrepancy in
the eigenvalues. For $r = 8$, we see that all
of the methods obtain similar eigenvalues.
From the preceding, it is unclear how to choose
the correct rank $r$ without a priori
knowledge.

\begin{table}
  \centering
  \label{tab:sst_eigs}
  \begin{tabular}{|c|c|c|c|}                              
\hline                                                  
Coefficient & optimized DMD & tlsDMD & exact DMD \\     
\hline                                                  
+1.5302e+04 & +2.8122e+18 & +Inf & +Inf \\              
\hline                                                  
+9.8238e+02 & -1.3565e+17 & +Inf & +Inf \\              
\hline                                                  
+9.7476e+02 & -3.6530e+02 & -3.6695e+02 & -3.6767e+02 \\
\hline                                                  
+9.7476e+02 & +3.6530e+02 & +3.6695e+02 & +3.6767e+02 \\
\hline                                                  
+2.3189e+02 & -6.5046e+02 & -5.0713e+02 & -5.3827e+02 \\
\hline                                                  
+2.3189e+02 & +6.5046e+02 & +5.0713e+02 & +5.3827e+02 \\
\hline                                                  
+2.1204e+02 & -1.8261e+02 & -1.8957e+02 & -1.9056e+02 \\
\hline                                                  
+2.1204e+02 & +1.8261e+02 & +1.8957e+02 & +1.9056e+02 \\
\hline                                                  
+1.3309e+02 & -7.9859e+02 & -8.1385e+02 & -8.5878e+02 \\
\hline                                                  
+1.3309e+02 & +7.9859e+02 & +8.1385e+02 & +8.5878e+02 \\
\hline                                                  
+1.1419e+02 & -2.9084e+03 & -4.2224e+03 & -6.6036e+03 \\
\hline                                                  
+1.1419e+02 & +2.9084e+03 & +4.2224e+03 & +6.6036e+03 \\
\hline                                                  
+6.6960e+01 & +1.6955e+03 & +1.9867e+03 & +1.9270e+03 \\
\hline                                                  
+6.6960e+01 & -1.6955e+03 & -1.9867e+03 & -1.9270e+03 \\
\hline                                                  
+4.8591e+01 & -1.0973e+03 & -1.5181e+03 & -1.7531e+03 \\
\hline                                                  
+4.8591e+01 & +1.0973e+03 & +1.5181e+03 & +1.7531e+03 \\
\hline                                                  
\end{tabular} 

  \caption{Example 4. This table shows the wavelength
    (in days) for each eigenvalue computed
    using the exact DMD, tlsDMD, and optimized DMD,
    for the evenly spaced data with $r = 16$.
    The optimized DMD wavelengths are ordered
    according to the magnitude of the corresponding
    spatial mode and the other wavelengths are chosen
    in the order which best matches the optimized
    DMD values.}
  
\end{table}

Because this data set comes from temperature
measurements over time, we know some of the
wavelengths we should find in the data set.
In particular, there should be a background mode
with infinite wavelength and a mode corresponding
to a tropical year (365.24 days).
In \cref{tab:sst_eigs}, we see that these wavelengths
are discovered by each DMD method. The
tropical year wavelength recovered by the
optimized DMD method (365.30 days) is
remarkably close to the true value, particularly
considering that the data is only provided
to four decimal digits. We see that the
second harmonic of this frequency, or a wavelength
of half a tropical year, is also discovered
by the DMD methods. The half-year wavelength
from the optimized DMD (182.61 days)
is again quite accurate (actual value 182.62
days).

\begin{table}
  \centering
  \label{tab:sst_eigs_evsu}
  \begin{tabular}{|c|c|c|c|c|}                                                             
\hline                                                                                   
even --- $b$ & uneven --- $b$ & even --- $\lambda$ & uneven --- $\lambda$ & projection \\
\hline                                                                                   
+1.5302e+04 & +1.4465e+04 & +2.8122e+18 & -1.9361e+20 & +9.9964e-01 \\                   
\hline                                                                                   
+9.8238e+02 & +5.4390e+02 & -1.3565e+17 & +4.3709e+17 & +4.4502e-02 \\                   
\hline                                                                                   
+9.7476e+02 & +9.7456e+02 & -3.6530e+02 & -3.6526e+02 & +9.9981e-01 \\                   
\hline                                                                                   
+9.7476e+02 & +9.7456e+02 & +3.6530e+02 & +3.6526e+02 & +9.9981e-01 \\                   
\hline                                                                                   
+2.3189e+02 & +2.8990e+02 & -6.5046e+02 & -6.0671e+02 & +8.6006e-01 \\                   
\hline                                                                                   
+2.3189e+02 & +2.8990e+02 & +6.5046e+02 & +6.0671e+02 & +8.6006e-01 \\                   
\hline                                                                                   
+2.1204e+02 & +2.1168e+02 & -1.8261e+02 & -1.8259e+02 & +9.9591e-01 \\                   
\hline                                                                                   
+2.1204e+02 & +2.1168e+02 & +1.8261e+02 & +1.8259e+02 & +9.9591e-01 \\                   
\hline                                                                                   
+1.3309e+02 & +8.6639e+01 & -7.9859e+02 & -8.6279e+02 & +8.7601e-01 \\                   
\hline                                                                                   
+1.3309e+02 & +8.6639e+01 & +7.9859e+02 & +8.6279e+02 & +8.7601e-01 \\                   
\hline                                                                                   
+1.1419e+02 & +9.6051e+01 & -2.9084e+03 & -2.2726e+03 & +7.6004e-01 \\                   
\hline                                                                                   
+1.1419e+02 & +9.6051e+01 & +2.9084e+03 & +2.2726e+03 & +7.6004e-01 \\                   
\hline                                                                                   
+6.6960e+01 & +7.7304e+01 & +1.6955e+03 & +1.4850e+03 & +8.0135e-01 \\                   
\hline                                                                                   
+6.6960e+01 & +7.7304e+01 & -1.6955e+03 & -1.4850e+03 & +8.0135e-01 \\                   
\hline                                                                                   
+4.8591e+01 & +1.0787e+02 & -1.0973e+03 & -1.2015e+03 & +7.5979e-01 \\                   
\hline                                                                                   
+4.8591e+01 & +1.0787e+02 & +1.0973e+03 & +1.2015e+03 & +7.5979e-01 \\                   
\hline                                                                                   
\end{tabular}                    

  \caption{Example 4. This table compares the
    wavelengths $\lambda$ obtained using the
    optimized DMD for each of the evenly spaced
    and randomly spaced data sets,
    with $r = 16$.
    The wavelengths obtained from the evenly spaced
    data are ordered
    according to the magnitude of the corresponding
    spatial mode (the $b$ values) and the
    wavelengths for the randomly spaced data are chosen
    in the order which best matches the evenly 
    spaced DMD values. In the last column, we
    report the cosine of the angle between
    the spatial mode from the evenly spaced
    data and the spatial mode from the randomly
    spaced data.}
  
\end{table}

\begin{figure}[h!]

\centering

\includegraphics[width=.8\textwidth]{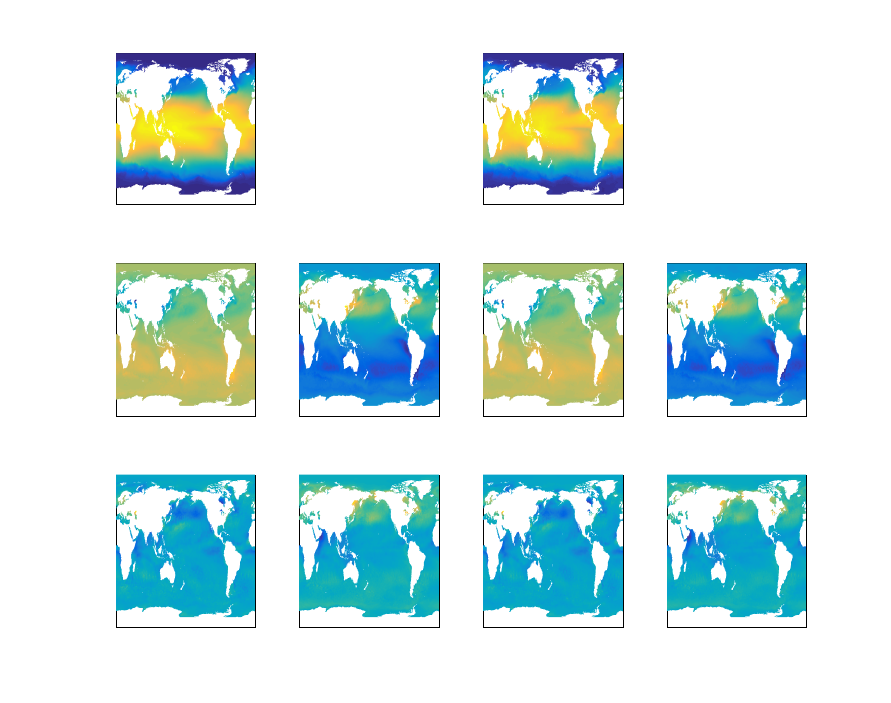}

\caption{Example 4. We plot the spatial modes obtained
  using the optimized DMD for two different subsets of the
  data, one with evenly spaced snapshots (left two columns) and
  randomly spaced snapshots (right two columns). The top
  row corresponds
  to a static background mode, the middle row the real and
  imaginary parts of a mode with a
  one-year wavelength, and the bottom row the real and imaginary
  parts of a mode with a half-year wavelength.}

\label{fig:sst_modes_evsu}

\end{figure}

One way to verify the eigenvalues
obtained from the optimized DMD on the evenly
spaced data set is to compare these with
the eigenvalues from the randomly spaced data
set. For $r = 16$, we computed DMD eigenvalues
and modes using the optimized DMD on each
data set.
The wavelengths corresponding to these
eigenvalues are reported
in \cref{tab:sst_eigs_evsu}. There is
good agreement for the infinite, one year,
and half-year wavelengths, and less-so
for the others. Further, the
spatial modes for these wavelengths are
very similar. We measure this using the cosine
of the angle between the corresponding
spatial modes, which is near
1 in absolute value for the infinite,
one year, and half-year wavelengths.
Heatmaps of these modes are provided
in \cref{fig:sst_modes_evsu}. We find this
to be a convincing confirmation of these
wavelengths, which did not require
a priori knowledge. 

\begin{table}
  \centering
  \label{tab:sst_times}
  \begin{tabular}{|c|c|c|c|}                       
\hline                                           
rank & $t_{DMD}$ & $t_{tlsDMD}$ & $t_{optDMD}$ \\
\hline                                           
2 & 2.7675e-01 & 2.7798e-01 & 1.3367e+00 \\      
\hline                                           
4 & 4.3028e-01 & 3.6161e-01 & 1.1965e+00 \\      
\hline                                           
8 & 7.1807e-01 & 3.6795e-01 & 1.0157e+00 \\      
\hline                                           
16 & 6.7264e-01 & 3.8975e-01 & 1.6571e+00 \\     
\hline                                           
32 & 1.0951e+00 & 1.2162e+00 & 4.7037e+00 \\     
\hline                                           
64 & 2.1410e+00 & 1.5664e+00 & 1.6101e+01 \\     
\hline                                           
128 & 5.4412e+00 & 3.4158e+00 & 3.0226e+01 \\    
\hline                                           
\end{tabular}  

  \caption{Example 4. The run time (excluding
    the SVD of the data used for projecting
    onto POD modes) for each method and various
    values of the rank $r$ on the
    evenly spaced data set.}
  
\end{table}

Some basic timing info for these
calculations is provided in
\cref{tab:sst_times}. The time reported is the
total time used by the algorithm, excluding
the cost of the SVD used to project the data
onto POD modes (the time for this calculation was
32 seconds). For ranks less than or equal to
$r = 32$, the optimized DMD costs only about 4 times
as much as the other methods. For the largest
rank, $r = 128$, the optimized DMD is about a factor
of 10 times more costly. Even then,
the cost of the optimized DMD is roughly
equal to the cost of the initial SVD.
For larger $r$, the computational cost of
the optimized DMD appears to increase like $r^2$,
which is lower than the bound we expect
based on the estimates in \cref{sec:algorithm}.

\section{Conclusions and future directions}

Based on the numerical experiments above, we 
believe that the optimized DMD is the DMD algorithm of choice
for many applications. The resulting modes and
eigenvalues are less sensitive to noise than those
computed using the other DMD methods we tested 
and the optimized DMD overcomes the bias issues 
of the exact DMD.   
In some cases, the improvement over existing methods
in robustness to noise is significant. For example
2, the optimized DMD algorithm is better able to 
capture the hidden dynamics than the other DMD methods,
sometimes showing an order of magnitude improvement in the
error. For the sea surface temperature data, example 4, 
the optimized DMD obtains modes which more accurately 
describe yearly patterns; indeed, the accuracy of the 
yearly and half-yearly wavelengths obtained by the optimized
DMD is comparable with the accuracy of the data
(the exact DMD and tlsDMD are an order of magnitude
less accurate for these wavelengths).

Of course, these advantages come
at a cost: computing the optimized DMD requires the
solution of a nonlinear, nonconvex optimization
problem. As noted above, it is unclear whether
we have actually solved this optimization problem
(globally) in our numerical experiments. 
The solutions we have obtained nonetheless represent 
improvements over existing DMD methods and the cost 
of the optimization algorithm
is modest for the range of problem sizes considered
above. We see in \cref{fig:senhid_times}
that, for this problem size, the cost of the optimized 
DMD is only about 1.5 times
that of the exact DMD in the worst case (as the number
of snapshots increases, the cost of the projection onto
POD modes begins to dominate the calculation, so the times
for each method become closer to equal). For the 
larger climate example, the run time of the optimized
DMD was about 6 times that of the exact DMD (for various
values of the reconstruction rank $r$), see
\cref{tab:sst_times}. This is
a more significant cost increase, but even for the
largest rank, $r = 128$, the cost of the optimized DMD
was roughly equal to that of the SVD required to
project onto POD modes (this cost is left out of the
values in \cref{tab:sst_times}). We should stress that
these timings depend strongly on the implementation of the
given algorithms and even the parameters sent to
the optimization routine. The MATLAB implementation of 
\cref{algo:optdmd,algo:approxoptdmd} we prepared for
these experiments is available online 
\cite{askham2017figcode,askham2017optdmd}.

The apparent efficiency of the optimized DMD algorithms
is a result of the variable projection methods which
have been developed for nonlinear least squares problems. 
Further, by rephrasing the problem as 
fitting exponentials to data, the optimized DMD 
also represents a more general method. It is no longer
necessary to assume that the snapshots are evenly
spaced in time. 

There are a few different avenues available for 
future research. As mentioned above, the variable
projection framework applies to a wide range of 
optimization problems and could therefore serve as
the basis for an optimized DMD with the addition
of a sparsity prior (see \cref{subsec:modvarpro}).
Such a method could help side-step the problem of 
choosing the correct target rank a priori. Also mentioned
above is the possibility of using a block Schur decomposition,
as opposed to an eigendecomposition, in the definition
of the DMD. This could potentially improve the ability
of the DMD to stably approximate transient dynamics. 
Finally, a few new directions are available because
the sample times need not be evenly spaced for 
the optimized DMD. As seen in miniature for the climate
example, making use of arbitrary sample times
allows for some interesting types of cross-validation
(and indeed expands the number of possible cross-validation
sets for a given set of snapshots).
There is also the possibility of 
using incoherent time sampling (e.g. randomly spaced
times) to detect high frequency signals using 
less data than implied by the Shannon sampling 
theorem. 

\section*{Acknowledgments} 

The authors would like to thank Professor Randall J. LeVeque
for pointing them to the inverse differential equations 
application in \cite{golub1979}. JNK acknowledges helpful 
and insightful conversations with Steven Brunton, Bingni Brunton, 
Joshua Proctor, Jonathan Tu and Clarence Rowley.


\begin{appendix} 

\end{appendix}

\bibliographystyle{siamplain}
\bibliography{optimized-dmd-2017}

\end{document}